\documentclass[11pt,a4paper]{article} % KOMA-Script article scrartcl

\usepackage[utf8]{inputenc}    % utf8 support       %!!!!!!!!!!!!!!!!!!!!
\usepackage[T1]{fontenc}       % code for pdf file  %!!!!!!!!!!!!!!!!!!!!
\usepackage{url}
\usepackage[nochapters]{classicthesis} % nochapters

\usepackage{enumerate}

\usepackage{xcolor}
\definecolor{brightmaroon}{rgb}{0.76, 0.13, 0.28}
\definecolor{maroon}{rgb}{0.5, 0, 0}
\definecolor{webgreen}{rgb}{0, 0.5, 0}

\usepackage{hyperref} % Required for adding links	and customizing them
\hypersetup{colorlinks, breaklinks, urlcolor=Maroon, linkcolor=Maroon, citecolor=webgreen} % Set link colors

\usepackage{xspace}

\RequirePackage[letterpaper, top=1in, bottom=1.5in, 
  left=1.5in, right=1.5in,showframe=false]{geometry}

\usepackage{mathtools}
\PassOptionsToPackage{%
    backend=bibtex, %for arxiv
	language=auto,%
	style=numeric-comp,%
    bibstyle=authoryear, %dashed=false, % dashed: substitute rep. author with ---
    sorting=nyt, % name, year, title
    maxcitenames=3, % default: 3, et al.
    mincitenames=1,
    natbib=true, % natbib compatibility mode (\citep and \citet still work)
    giveninits=true,
    bibstyle=numeric,
    doi=false,
    isbn=false,
    url=false,
}{biblatex}
\usepackage{biblatex}
\renewbibmacro*{bbx:savehash}{}
\DeclareFieldFormat{postnote}{#1}

\usepackage{amsmath}
\usepackage{amsthm}
\usepackage{array}

\usepackage{bbold} % mathbb

\newcommand{\E}[1]{{\mathbb E}\left[#1\right]}
\newcommand{\e}{{\mathbb E}}
\newcommand{\V}[1]{{\mathrm{Var}}\left(#1\right)}

\newcommand{\Cov}[1]{{\mathrm{Cov}}\left(#1\right)}

\newcommand{\p}[1]{{\mathbb P}\left\{#1\right\}}

\newcommand\inlawHIGH{\,{\buildrel d \over \rightarrow}\,} 
\newcommand\inlaw{{\inlawHIGH}}

\newcommand{\eql}{\,{\buildrel \cL \over =}\,}

\newcommand{\eqd}{\,{\buildrel {\mathrm{def}} \over =}\,}

\newcommand{\Gam}{\mathop{\mathrm{Gamma}}}

\newcommand{\iid}{i.i.d.\@\xspace}

\usepackage{mathtools}
\usepackage{stmaryrd}

\newcommand{\dsR}{{\mathbb R}}

\newcommand{\dsN}{{\mathbb N}}

\newcommand\cK{{\cal K}}
\newcommand\cL{{\cal L}}

\newcommand\cX{{\cal X}}

\DeclarePairedDelimiter{\floor}{\lfloor}{\rfloor}

\DeclarePairedDelimiter{\fracPart}{\lbrace}{\rbrace}

\newcommand\bigO[1]{{O\left( #1 \right)}}
\newcommand\Op[1]{{O_{p}\left( #1 \right)}}
\newcommand\smallo[1]{{o\left( #1 \right)}}
\newcommand\op[1]{{o_{p}\left( #1 \right)}}

\newcommand\numberthis{\addtocounter{equation}{1}\tag{\theequation}}

\newcommand{\mySeq}[2]{(#1)_{#2}}
\newcommand{\mySeqBig}[2]{(#1,#2)}

\renewcommand{\Re}[1]{\mathrm{Re}(#1)}

\newtheorem{lemma}{Lemma}

\newtheorem{theorem}{Theorem}

\newtheorem{proposition}{Proposition}
\newtheorem{corollary}{Corollary}

\theoremstyle{definition}

\newtheorem{myRemark}{Remark}
\newtheorem*{myRemark*}{Remark}

\numberwithin{theorem}{section}
\numberwithin{lemma}{section}
\numberwithin{corollary}{section}
\numberwithin{proposition}{section}
\numberwithin{myRemark}{section}

\numberwithin{equation}{section}

\newcommand{\ee}{{\mathrm e}}
\newcommand{\ii}{{\mathrm i}}
\newcommand{\dd}{{\mathrm d}}

\usepackage{textcomp}

\usepackage{eufrak}

\usepackage{tikz}
\usetikzlibrary{decorations.pathreplacing,angles,quotes}
\usetikzlibrary{shapes.geometric}
\usepackage[skip=2pt]{caption}

\allowdisplaybreaks

\let\originalleft\left
\let\originalright\right
\renewcommand{\left}{\mathopen{}\mathclose\bgroup\originalleft}
\renewcommand{\right}{\aftergroup\egroup\originalright}

\newcommand{\kcut}{\cK}

\newcommand{\kcutE}{\cK^{e}}

\newcommand{\alarm}{E}
\newcommand{\alarmRV}[2]{\alarm_{#1,#2}}

\newcommand{\clock}{T}
\newcommand{\clockRV}[2]{\clock_{#1,#2}}
\newcommand{\clockrv}{\clockRV{r}{v}}
\newcommand{\clockkv}{\clockRV{k}{v}}

\newcommand{\record}{I}
\newcommand{\recordRV}[2]{\record_{#1,#2}}
\newcommand{\recordrv}{\record_{r,v}}

\newcommand{\rec}{\cX}
\newcommand{\recnR}[1]{\rec_{n}^{#1}}
\newcommand{\recnyR}[1]{\rec_{n,y}^{#1}}
\newcommand{\recnyr}{\recnyR{r}}
\newcommand{\recnall}{\rec_{n}}
\newcommand{\recnalle}{\rec_{n}^{e}}
\newcommand{\recni}{\recnR{1}}
\newcommand{\recnr}{\recnR{r}}
\newcommand{\recneR}[1]{\rec_{n}^{e,#1}}

\newcommand{\Wrab}{W_{r,k,\gamma}}
\newcommand{\Wii}{W_{1,1,\gamma}}
\newcommand{\nurab}{\nu_{r,k,\gamma}}

\newcommand{\rescale}[1]{\frac{\lg(n)^{\frac{#1}{k}+1}}{C_{2}(#1) n}}
\newcommand{\rescaleInv}[1]{\frac{C_{2}(#1) n}{\lg(n)^{\frac{#1}{k}+1}}}

\newcommand{\mbar}{\overline{m}}
\newcommand{\mpr}{\hat{m}}
\newcommand{\psibar}{\bar{\psi}}

\newcommand{\Tbin}{T^{\mathrm{bin}}}
\newcommand{\Tbinn}{\Tbin_{n}}

\newcommand{\xirv}{\xi_{r,v}}

\begin{document}

\pagestyle{plain}
\title{\rmfamily\normalfont\spacedallcaps{Cutting resilient networks -- complete binary trees}}
\author{
    Xing Shi Cai, Cecilia Holmgren\thanks{This work was
        partially supported by two grants from the Knut and Alice Wallenberg Foundation, a grant
        from the Swedish Research Council, and the Swedish Foundations' starting grant from the
        Ragnar S\"{o}derberg Foundation. Emails:
    \texttt{\{xingshi.cai, cecilia.holmgren\}@math.uu.se}.}\\
    \footnotesize{Uppsala University, Uppsala, Sweden}
}

\maketitle

\begin{abstract}
    In our previous work \cite{Cai010}, we introduced the random \(k\)-cut number for rooted graphs.
    In this paper, we show that the distribution of the \(k\)-cut number in complete binary trees
    of size \(n\), after rescaling, is asymptotically a periodic function of \(\lg n - \lg \lg n\).
    Thus there are different limit distributions for different subsequences, where these limits are
    similar to weakly \(1\)-stable distributions. This generalizes the result
    for the case \(k=1\), i.e., the traditional cutting model, by \citet{janson04}.
\end{abstract}

\section{Introduction}

\subsection{The model and the motivation}

In our previous work \cite{Cai010}, we introduced the \(k\)-cut number for rooted graphs.  Let \(k\)
be an integer.  Let \(G_{n}\) be a connected graph of \(n\) nodes with exactly one node labelled as
the root.  We remove nodes from the graph using this random procedure: (note that in our model nodes
are only removed after having been cut \(k\) times)
\begin{enumerate}
    \item Initially set every node's cut-counter to zero, i.e., no node has ever been cut.
    \item Choose one node uniformly at random from the component containing the root and increase
        its cut-counter by one, i.e., we cut the selected node once.
    \item If this node's cut-counter hits \(k\), i.e., it has been cut \(k\) times, then remove it from the graph.
    \item If the root has been removed, then stop. Otherwise, go to step 2.
\end{enumerate}
We call the (random) total number of cuts needed for this procedure to end the \(k\)-cut number and
denote it by \(\kcut(G_{n})\).  The traditional cutting model corresponds to the case that \(k=1\).

We can also cut and remove edges instead of nodes using the same process with the modification that
we stop when the root has been isolated. We denote the total number of cuts needed for this edge
removing process to end by \(\kcutE(G_{n})\).

The \(k\)-cut number can be seen as a measure of the difficulty for the destruction of a resilient
network. For example, in a botnet, a bot-master controls a large number of compromised computer
(bots) for various cybercrimes. To counter attack a botnet means to reduce the number of bots
reachable from the bot-master by fixing compromised computers \cite{4413000}.  We can view a botnet
as a graph and fixing a computer as removing a node from the graph.  If we assume that each
compromised computer takes \(k\)-attempts to clean, and each attempt aims at a computer chosen
uniformly at random, then the \(k\)-cut number is precisely the number of attempts of cleaning up
needed to completely destroy a botnet.

The case \(k=1\), i.e., the traditional cutting model has been well-studied. It was first introduced
by \citet{meir70} for uniform random Cayley trees. \citet{janson04, janson06} studied one-cuts in
binary trees and conditioned Galton-Watson trees. \citet{ab14} simplified the proof for the limit
distribution of one-cuts in conditioned Galton-Watson trees. The cutting model has also been studied
in random recursive trees, see \citet{meir74}, \citet{iksanov07}, and \citet{drmota09}. For binary
search trees and split trees, see \citet{holmgren10, holmgren11}.

In our previous work \cite{Cai010}, we mainly analyzed \(\kcut(P_{n})\), the \(k\)-cut number for a path of length
\(n\), which generalizes the record number in a uniform random permutation. In this paper, we
continue our investigation in \emph{complete binary trees}, i.e., binary trees in which each level
is full except possibly for the last level, and the nodes at the last level occupy the leftmost
positions. If the last level is also full, then we call the tree a \emph{full binary tree}.

\subsection{An equivalent model}

Let \(\Tbinn\) be a complete binary tree of size \(n\). Let \(\recnall \eqd \kcut(\Tbinn)\) and
\(\recnalle \eqd \kcutE(\Tbinn)\) with the root of the tree as the root of the graph. There is an
equivalent way to define \(\recnall\).  Let \(\mySeqBig{\alarmRV{r}{v}}{r
\ge 1, v \in \Tbinn}\) be \iid{} exponential random variables with mean \(1\).  Let
\(\clockRV{r}{v}\eqd\sum_{j=1}^{r} \alarmRV{j}{v}\).  Imagine each node in \(\Tbinn\) has an alarm
clock and node \(v\)'s clock fires at times \(\mySeqBig{\clockRV{r}{v}}{r \ge 1}\). If we cut a node
when its alarm clock fires, then due to the memoryless property of exponential random variables, we
are actually choosing a node uniformly at random to cut. 

However, this also means that we are cutting nodes that have already been removed from the tree. Thus for
a cut on node \(v\) at time \(\clockRV{r}{v}\) (for some \(r \le k\)) to be counted in \(\recnall\),
none of its ancestors can have already been cut \(k\) times, i.e.,
\begin{equation}\label{eq:r:rec}
    \clockRV{r}{v} < \min_{u: u \prec v} \clockRV{k}{u},
\end{equation}
where \(u \prec v\) denotes that \(u\) is an ancestor of \(v\).  When the event in \eqref{eq:r:rec} happens, we say that
\(T_{r,v}\) (or simply \(v\)) is an \emph{\(r\)-record} and let \(\recordrv\) be the indicator random variable for 
this event. Let \(\recnR{r}\) be the total number of \(r\)-records, i.e.,
\(\recnR{r}\eqd\sum_{v} \recordrv\). Then obviously \(\recnall \eql \sum_{r=1}^{k}
\recnR{r}\). We use this equivalence for the rest of the paper.

By assigning alarm clocks to edges instead of nodes, we can define the edge version of
\(r\)-records \(\recneR{r}\) and have \(\recnalle\eql \sum_{r=1}^{k} \recneR{r}\).

\subsection{The main results}

To introduce the main results, we need some notations.
Let \(\fracPart{x}\) denote the fractional part of \(x\), i.e., \(\fracPart{x}\eqd x - \floor{x}\).
Let \(\Gamma(a)\) be the Gamma function \cite[5.2.1]{DLMF}.  Let \(\Gamma(a,x)\) be the upper
incomplete Gamma function \cite[8.2.2]{DLMF}. Let \(Q(a,x)\eqd \Gamma(a,x)/\Gamma(a)\).  Let
\(Q^{-1}(a,x)\) be the inverse of \(Q(a,x)\) with respect to \(x\). Let \(\lg(x)\eqd \log_{2}(x)\).

\begin{theorem}\label{thm:1:rec}
    Assume that \(\fracPart{\lg n - \lg \lg n} \to \gamma \in [0,1]\) as \(n \to \infty\).
    Then
    \begin{equation}\label{eq:rec:1:dist}
        \rescale{r}
        \recnr 
        -\mu_{r,n}
        \inlaw 1- C_{3}(r) \Wrab
        ,
    \end{equation}
    where 
    \begin{equation}\label{eq:mu:new}
        \mu_{r,n}
        =
        \frac{k}{r} \lg (n)
        +
        \sum _{i=1}^k 
        C_{1}(r,i)
        \lg (n)^{1-\frac{i}{k}} 
        +
        \lg (\lg (n))
        ,
    \end{equation}
    \(C_{1}(\cdot, \cdot)\), 
    \(C_{2}(\cdot, \cdot)\), 
    and
    \(C_{3}(\cdot)\)
    are constants defined in \autoref{lem:rescale}, and \(\Wrab\)
    has an infinitely divisible distribution with the characteristic function
    \begin{equation}\label{eq:Wrab}
        \E{\exp\left({\ii t \Wrab}\right)}
        =
        \exp\left( 
            \ii f_{r,k,\gamma} t + 
            \int_{0}^{\infty}
            \left( 
                \ee^{\ii t x}
                -1
                -\ii
                t x
                \cdot
                \mathbb 1 [x<1]
            \right)
            \,
            \mathrm{d}\nurab(x)
        \right)
        ,
    \end{equation}
    where \(f_{r,k,\gamma}\) is a constant defined later in \eqref{eq:f:gamma} and the L\'evy measure
    \(\nurab\) has support on \((0,\infty)\) with
    density
    \begin{equation}\label{eq:nurab}
        \begin{aligned}
            \frac{\dd\nurab}{\dd x}
            =
            \frac{\Gamma \left(\frac{r}{k}\right)^2 }{x^2}
            \sum_{s \ge 1}
            &
            4^{\fracPart{\gamma +{\lg \left({x}/{\Gamma \left(\frac{r}{k}\right)}\right)}}- s}
            \exp \left(Q^{-1}\left(\frac{r}{k},2^{\fracPart{\gamma +{\lg \left({x}/{\Gamma \left(\frac{r}{k}\right)}\right)}}-s}\right)\right)
            \\
            &
            \qquad
            Q^{-1}\left(\frac{r}{k},2^{\fracPart{\gamma +{\lg \left({x}/{\Gamma
                                    \left(\frac{r}{k}\right)}\right)}}-s}\right)^{1-\frac{r}{k}} 
            .
        \end{aligned}
    \end{equation}
\end{theorem}

\begin{theorem}\label{thm:mean}
    Assume the same conditions as in \autoref{thm:1:rec}.
    Then
    \begin{equation}\label{eq:rec:r:dist}
        \rescale{1}
        \left(
            \recnall
            -
            \sum_{r=1}^{k}
            \rescaleInv{r}
            \mu_{r,n}
        \right)
        \inlaw 1- C_{3}(1) W_{1,k,\gamma }
        .
    \end{equation}
    The same holds true for \(\recnalle\).
\end{theorem}

\begin{myRemark}
    Let \(\widetilde{X}_{n}\) denote the left-hand-side of \eqref{eq:rec:r:dist}.  Another way of
    formulating \autoref{thm:mean} is by saying that the distance, e.g., in the L\'evy metric,
    between the distribution of \(\widetilde{X}_{n}\) and the distribution of
    \(1-C_{3}(1)W_{1,k,\fracPart{\lg n - \lg \lg n}}\) tends to zero as \(n \to \infty\).
\end{myRemark}

\begin{myRemark}
We do not have a closed form for \(C_{1}(\cdot, \cdot)\). But for specific \(k\) they are easy to compute 
with computer algebra systems. 
When \(k=r=1\), i.e., when \(\recni=\recnall\), \eqref{eq:rec:r:dist} reduces to
\begin{equation}\label{eq:rec:1:dist:k1}
    \recnall \frac{\lg (n)^2}{n}-\lg (n)-\lg (\lg (n)) \inlaw -W_{1,1,\gamma }
    ,
\end{equation}
and since \(Q^{-1}(1,x)=\log(1/x)\), \eqref{eq:nurab} reduces to
\begin{equation}\label{eq:nurab:k:1}
    \frac{\dd \nu_{1,1,\gamma}}{\dd x}
    =
    \frac{
        1
    }
    {x^2}
    2^{\fracPart{\lg x + \gamma}}.
\end{equation}
In other words, we recover the result for the traditional cutting model in complete binary trees by
\citet[Theorem 1.1]{janson04}.
When \(k=2\), \eqref{eq:rec:r:dist} reduces to
\begin{equation}\label{eq:rec:1:dist:k2}
    \sqrt{
        \frac{8}{\pi}
    }
    \frac{\lg(n)^{\frac{3}{2}}}{n}
    \recnall
    - 2 \lg(n)
    -\frac{1}{3} \sqrt{\frac{2}{\pi}} \lg(n)^{\frac{1}{2}}
    - \lg(\lg(n))
    - \frac{11}{3}
    \inlaw
    -\frac{2 W_{1,2,\gamma}}{\sqrt{\pi}}
    .
\end{equation}
\end{myRemark}

\begin{myRemark}
    In Remark 1.5 of \cite{janson04}, \citeauthor{janson04} mentioned that when \(k=r=1\), if \(\Wii'\)
    and \(\Wii''\) are independent copies of \(\Wii\), then \(\Wii'+\Wii'' \eql 2 \Wii +2\), but
    the corresponding statement for three copies of \(\Wii\) is false. In other words, \(\Wii\) is
    roughly similar to a \(1\)-stable distribution. This extends to general \(k\) in the sense that
    \begin{equation}\label{eq:two:copies}
        \Wrab'+\Wrab'' \eql 2 \Wrab +2 \int_{1}^{2} x \,\mathrm{d}\nurab(x),
    \end{equation}
    with \(\int_{1}^{2} x \,\mathrm{d}\nu_{1,1,\gamma}(x) = 1\).
    This follows by computing the characteristic functions of both sides using \eqref{eq:Wrab} and
    by noticing that
    \begin{equation}\label{eq:periodic}
        \left. \frac{\text{d$\nurab $}}{\text{d}x} \right|_{x=u}
        =
        \frac{1}{4} \left.\frac{\text{d$\nurab $}}{\text{d}x}\right|_{x=\frac{u}{2}}
        .
    \end{equation}
\end{myRemark}

In the rest of the paper, we will first compute the expected number and variance of \(r\)-records
conditioning on \(\clockRV{k}{o}=y\), where \(o\) denotes the root. Then we show that the fluctuation of the total number of
\(r\)-records from its mean is more or less the same as the sum of such fluctuations in each subtree
rooted at height \(L\eqd\floor*{\left( 2-\frac{1}{2 k} \right) \lg \lg n}\), conditioning on what happens below height \(L\). This sum can be further
approximated by a sum of independent random variables. Finally, we apply a classic theorem regarding
the convergence to infinitely divisible distributions by \citet[Theorem 15.23]{kallenberg02} to prove
\autoref{thm:1:rec} and \autoref{thm:mean}.

The proof follows a similar path as \citet{janson04} did for the case \(k=1\). However, the analysis
for \(k \ge 2\) is significantly more complicated.

\citet{holmgren10, holmgren11} showed that when \(k=1\), \(\recnall\) has similar behaviour in
binary search trees and split trees as in complete binary trees. We are currently trying to prove
this for \(k \ge 2\).

\section{Some more notations}

We collect some of the notations which are used frequently in this paper.

Let \(\Gamma(a)\) be the Gamma function \cite[5.2.1]{DLMF}, i.e.,
\begin{equation}\label{eq:gam:def}
    \Gamma\left(a\right)=\int_{0}^{\infty}
    \ee^{-t}
    t^{a-1}
    \mathrm{d}t,
    \qquad
    \Re{a}> 0.
\end{equation}
Note that \(\Gamma(k)=k!\) for \(k \in \dsN\).
Let \(\Gamma(a,x)\) and \(\gamma(a,x)\) be the upper and lower incomplete Gamma functions
respectively \cite[8.2]{DLMF}, i.e.,
\begin{equation}\label{eq:gam:up}
    \Gamma\left(a,z\right)=\int_{z}^{\infty}\ee^{-t}t^{a-1}\mathrm{d}t,
    \qquad
    \gamma\left(a,z\right)=\int_{0}^{z}\ee^{-t}t^{a-1}\mathrm{d}t,
    \qquad
    \Re{a}>0.
\end{equation}
Thus \(\gamma(a,x)\eqd\Gamma(a)-\Gamma(a,x)\).  Let \(\Gamma(a,x_{0},x_{1}) \eqd
    \Gamma(a,x_{0})-\Gamma(a,x_{1})\).  We also define
\(\gamma(a,\infty)\eqd\lim_{x\to\infty}\gamma(a,x)=\Gamma(a)\).

Let \(Q(a,x)\eqd \Gamma(a,x)/\Gamma(a)\).
Let \(Q^{-1}(a,x)\) be the inverse of \(Q(a,x)\).
Note that \(Q(1,x)=\ee^{-x}\) and \(Q^{-1}\left( 1,x \right)=\log\left( {1}/{x} \right)\).

Let \(m\) be the height of a complete binary tree of \(n\) nodes, i.e., \(m \eqd \floor{\lg n}\). Let
\(\ell \eqd \floor{\lg \lg n}\).  Let \(L \eqd \floor*{\left( 2-\frac{1}{2 k} \right) \lg \lg n}\).

For node \(v \in \Tbinn\), let \(h(v)\) be the height of \(v\), i.e., the distance (number of edges)
between \(v\) and the root, which we denote by \(o\). 

Let \(\recnyr\) be \(\recnR{r}-1\) conditioned on \(\clockRV{k}{o}=y\), i.e., the number of
\(r\)-record, excluding the root, conditioned on that the root is removed (cut the \(k\)-th time) at time \(y\).

For functions \(f:A\to\dsR\) and \(g:A\to\dsR\), we write \(f=\bigO{g}\) \emph{uniformly}
on \(B \subseteq A\) to indicate that there exists a constant \(C_{0}\) such that \(|f(a)|\le C_{0}
    |g(a)|\) for all \(a\in B\). The word \emph{uniformly} stresses that \(C_{0}\) does not depend on
\(a\).

We use the notation \(\Op{\cdot}\) and \(\op{\cdot}\) in the usual sense, see \cite{janson11}.

The notations \(C_{1}(\cdots), C_{2}(\cdots),\dots\) denote constants that depend on \(k\) and other
parameters but do not depend on \(n\).

\section{The expectation and the variance}

\begin{lemma}\label{lem:fmexpand}
    There exist constants \(\mySeq{C_{5}(j,b)}{j \ge 1, b \ge k+1}\) such that
    \begin{equation}\label{eq:fmexpand}
        \exp\left({\frac{m x^k}{k!}} \right)Q(k,x)^m
        =
        1
        +
        \sum _{j=1}^k \sum _{b=j k+j}^{j k+k} C_{5}(j,b) m^j x^b 
        +
        \bigO{m^{k+1} x^{(k+1)^2}+m x^{2 k+1}},
    \end{equation}
    uniformly for all \(x \in \left(0,m^{-k_{0}}\right)\), where 
    \(k_{0}\eqd\frac{1}{2} \left(\frac{1}{k}+\frac{1}{k+1}\right)\).
\end{lemma}

\begin{myRemark}
    We do not have a closed form for the constants \(C_{5}(j,b)\), but they are the coefficients of
    \(m^{j}x^{b}\) in \eqref{eq:fmexpand}. For fixed \(k\), they are easy
to find with computer algebra systems. For example, 
when \(k=1\), \eqref{eq:fmexpand} reduces to
\begin{equation}\label{eq:fmexpand:k1}
    \ee^{m x} Q(1,x)^{m}
    =
    1+
    \bigO{m^{2} x^{4}+m x^{3}},
\end{equation}
which is trivially true since \(Q(1,x)=\ee^{-x}\).
When \(k=2\), \eqref{eq:fmexpand} reduces to
\begin{equation}\label{eq:fmexpand:k2}
    \exp\left({\frac{m x^2}{2}}  \right) Q(2,x)^m=
    1
    +
    \frac{1}{3} m x^3 
    -
    \frac{1}{4}
    m x^4 
    +
    \frac{1}{18} m^2 x^6 
    +
    \bigO{ m^3 x^9+ m x^5}
    .
\end{equation}
\end{myRemark}

\begin{proof}
    Using the series expansion of \(Q(k,x)\) given by \cite[8.7.3]{DLMF}, it is easy to verify that
    \begin{equation}\label{eq:Q:e}
        \left( 
        \exp
        \left( 
            \frac{x^k}{k!}
        \right)
        Q\left( k,x \right)
        \right)^{m}
        =
        \left( 
        1
        -\sum _{j=1}^k 
        \frac{x^k (-x)^j }{(k-1)! j! (k+j)}
        -\frac{x^{2 k}}{2 (k!)^2}
        +\bigO{x^{2 k+1}}
        \right)^{m}
        ,
    \end{equation}
    uniformly for \(x \in (0, m^{-k_{0}})\).
    Taking the binomial expansion of the right-hand-side and ignoring
    small order terms gives \eqref{eq:fmexpand:k2}.
\end{proof}

\begin{lemma}\label{lem:expectation:full}
    In the case that the tree is full, i.e., \(n=2^{m+1}-1\), then
    \begin{equation}\label{eq:y:full}
            \e \recnyR{r}
            =
            2^{m+1}
            \left(
            \psi_{r}(m,y,2)
            +
            \bigO{
                m^{-\frac{1+r}{k}-1}
            }
            \right)
            ,
    \end{equation}
    where 
    \begin{equation}
        \begin{aligned}
            \psi_{r}(m,z,c)
            \eqd
            &
            \frac{m^{-\frac{r}{k}} }{r!}
            \left(
                \frac{\left({k!}\right)^{\frac{r}{k}}}{k} \gamma \left(\frac{r}{k},\frac{m z^k}{k!}\right)
                + 
                c 
                \frac{\left({k!}\right)^{\frac{r}{k}}}{k} m^{-1} \gamma
                \left(\frac{r+k}{k},\frac{m z^k}{k!}\right)
            \right.
            \\
            &
                +
                \sum _{j=1}^k \left(\sum _{b=j k+j}^{j k+k} 
                    \frac{(k!)^{\frac{b+r}{k}}}{k}
                    C_{6}(j,b) m^{j-\frac{b}{k}} \gamma
                    \left(\frac{b+r}{k},\frac{m z^k}{k!}\right)\right)
            \\
            &
            \left.
                + 
                \sum _{i=1}^k \frac{(-1)^i \left({k!}\right)^{\frac{i+r}{k}}}{k i!}
                m^{-\frac{i}{k}}  \gamma \left(\frac{i+r}{k},\frac{m z^k}{k!}\right)
            \right)
            ,
            \label{eq:phi:p}
        \end{aligned}
    \end{equation}
    where the implicit constants \(C_{6}(j,b)\) are defined in \eqref{eq:h0:1}.
\end{lemma}

\begin{proof}
    Let \(v\) be a node of height \(i\). For \(v\) to be an \(r\)-record, conditioning on
    \(\clockRV{k}{o}=y\), we need \(\clockrv<y\) and \(\clockRV{k}{u}>\clockrv\) for every \(u\) that
    is an ancestor of \(v\). Recall that 
    \(\clockRV{r}{v}\eqd\sum_{j=1}^{r} \alarmRV{j}{v}\), where \(\alarmRV{j}{v}\) are \iid
    exponential \(1\) random variables.
    Thus \(\clockRV{k}{u}\) are \iid \(\Gam(k,1)\) random
    variables and  \(\clockrv\) is a \(\Gam(r,1)\) random variable, which are independent from everything else.
    (See Theorem 2.1.12 of \cite{durrett10} for the relation between exponential distributions and
    Gamma distributions.)

    The Gamma distribution \(\Gam(r,1)\) has the density function
    \begin{equation}\label{eq:gam:density}
        g_{r}(x)=
        \begin{cases}
        \dfrac{x^{r-1}\ee^{-x}}{\Gamma(r)}
        & \text{if $x \ge 0$}
        ,
        \\
        0
        & \text{if $x < 0$}
        ,
        \end{cases}
    \end{equation}
    which implies \(\p{\Gam(r,1)>x}=Q(r,x)\).
    Thus,
    \begin{equation}\label{eq:i:prob}
        \begin{aligned}
        \E{\recordRV{r}{v} | \clockRV{k}{o}=y}
        &
        =
        \int_{0}^{y} g_{r}(x) \p{\Gam(k,1)>x}^{i-1}\,\dd x
        \\
        &
        =
        \int_{0}^{y} \frac{x^{r-1}\ee^{-x}}{\Gamma(r)} Q(k,x)^{i-1}\,\dd x
        .
        \end{aligned}
    \end{equation}
    When the tree is full, each level \(i\) has \(2^{i}\) nodes. Thus in this case
    \begin{equation}\label{eq:y:full:0}
        \begin{aligned}
            \e \recnyR{r}
            &
            =
            \sum_{i = 1}^{m} 
            2^{i} 
            \int_{0}^{y} \frac{x^{r-1}\ee^{-x}}{\Gamma(r)} Q(k,x)^{i-1}\,\dd x
            \\
            &
            =
            \int_{0}^{y} 
            2
            \frac{x^{r-1}\ee^{-x}}{\Gamma(r)} 
            \left(  
            \sum_{i = 1}^{m} 
            (2 Q(k,x))^{i-1}
            \right)
            \,\dd x
            \\
            &
            =
            \frac{2^{m+1}}{r!}
            \int_{0}^{y}
            x^{r-1} 
            \ee^{-\frac{ m x^{k}}{k!}}
            h_{0}(x)
            \left(\ee^{\frac{ x^{k}}{k!}} Q(k,x)
            \right)^{m}
            \,\dd x
            +
            \bigO{1}
            ,
        \end{aligned}
    \end{equation}
    where
    \begin{equation}\label{eq:h0}
        h_{0}(x)
        \eqd
        \frac{\ee^{-x}}{(2 Q(k,x)-1)} 
        =
        1
        +
        \frac{2 x^{k}}{k!}
        +
        \sum_{i=1}^{k}\frac{(-1)^i x^i}{i!}
        +
        \bigO{x^{k+1}}
        ,
    \end{equation}
    as \(x \to 0\) by \cite[8.7.3]{DLMF}.
    Thus uniformly for \(0 < x \le m^{-k_{0}}\) with \(k_{0}\eqd\frac{1}{2}\left(
            \frac{1}{k}+\frac{1}{k+1} \right)\),
    \begin{equation}\label{eq:h0:1}
        \begin{aligned}
            h_{0}(x)
            \left(\ee^{\frac{ x^{k}}{k!}} Q(k,x)
            \right)^{m}
            =
            &
            1
            +
            \frac{2 x^k}{k!}
            +
            \sum_{i=1}^{k} \frac{(-1)^i x^i}{i!}
            +
            \sum _{j=1}^k \left(\sum _{b=j k+j}^{j k+k} x^b m^j C_{6}(j,b)\right)
            \\
            &
            +
            \bigO{x^{k+1}+m^{k+1} x^{(k+1)^2}+ m x^{2 k+1}}
            ,
        \end{aligned}
    \end{equation}
    for some constants \(C_{6}(j,b)\),
    where we expand the left-hand-side using \eqref{eq:h0} and \autoref{lem:fmexpand}, and then omit
    small order terms.

    Note that for \(b \ge 0\) and \(j \ge 0\),
    \begin{equation}\label{eq:int:1}
        \int_0^y \exp\left(-{\frac{m x^k}{k!}}  \right) x^{r-1} x^b m^j  \, \dd x
        =
        \frac{(k!)^{\frac{b+r}{k}}}{k}
        m^{j-\frac{b+r}{k}}\gamma \left(\frac{b+r}{k},\frac{m y^k}{k!}\right)
        .
    \end{equation}
    Thus if \(y < m^{-k_{0}}\),
    by putting the expansion \eqref{eq:h0:1} into \eqref{eq:y:full:0} and
    integrating term by term, we get \eqref{eq:y:full}.

    For \(y>m^{-k_{0}}\), it is not difficult to verify that the part of the integral in
    \eqref{eq:y:full:0} over \([m^{-k_{0}},y]\) and the difference
    \(\psi_{r}(m,y,2)-\psi_{r}(m,m^{-k_{0}},2)\) are both exponentially small and can be absorbed by the error term.
\end{proof}

\begin{lemma}\label{lem:expectation:m}
    If \(h(v)=m\), then
    \begin{equation}
        \begin{aligned}
        \E{\recordRV{r}{v} | \clockRV{k}{o}=y}
        = 
        \psi_{r}(m,y,1)
        +
        \bigO{
            m^{-\frac{1+r}{k}-1}
        }
        =
        \psi_{r}^{*}(m,y)
        +
        \bigO{
            m^{-\frac{1+r}{k}}
        }
        ,
        \end{aligned}
        \label{eq:y:last}
    \end{equation}
    where
    \begin{equation}\label{eq:psi:main}
        \psi_{r}^{*}(m,y)
        \eqd
        \frac{m^{-\frac{r}{k}} }{r!}
        \frac{\left({k!}\right)^{\frac{r}{k}}}{k} \gamma \left(\frac{r}{k},\frac{m y^k}{k!}\right)
        .
    \end{equation}
\end{lemma}

\begin{proof}
    When \(v\) is a node of height \(m\), by \eqref{eq:i:prob},
    \begin{equation}\label{eq:i:prob:m}
        \begin{aligned}
            \E{\recordRV{r}{v} | \clockRV{k}{o}=y}
            &
            =
            \int_{0}^{y} \frac{x^{r-1}\ee^{-x}}{\Gamma(r)} Q(k,x)^{m-1}\,\dd x
            \\
            &
            =
            \frac{1}{\Gamma(r)}
            \int_{0}^{y}
            x^{r-1} 
            \frac{x^{r-1}\ee^{-x}}{\Gamma(r)}
            h_{2}(x)
            \left(\ee^{\frac{ x^{k}}{k!}}
                Q(k,x)
            \right)^{m}
            \,\dd x
            ,
        \end{aligned}
    \end{equation}
    where
    \(h_{2}(x) \eqd\frac{\ee^{-x}}{Q(k,x)}.\)
    Expanding \(h_{2}(x)\) by \cite[8.7.3]{DLMF} and using \autoref{lem:fmexpand}, we have, uniformly for \(x \in (0, m^{-k_0})\)
    with \(k_{0}\eqd\frac{1}{2}\left( \frac{1}{k}+\frac{1}{k+1} \right)\)
    \begin{equation}\label{eq:h2:1}
        \begin{aligned}
            h_{2}(x)
            \left(\ee^{\frac{ x^{k}}{k!}} Q(k,x)
            \right)^{m}
            =
            &
            1
            +
            \frac{x^k}{k!}
            +
            \sum_{i=1}^{k} \frac{(-1)^i x^i}{i!}
            +
            \sum _{j=1}^k \left(\sum _{b=j k+j}^{j k+k} x^b m^j C_{6}(j,b)\right)
            \\
            &
            +
            \bigO{x^{k+1}+m^{k+1} x^{(k+1)^2}+ m x^{2 k+1}}
            .
        \end{aligned}
    \end{equation}
    Note that this differs from \eqref{eq:h0:1} only by the constant in the term \(x^{k}/k!\).
    Thus the first equality in \eqref{eq:y:last} follows as in \autoref{lem:expectation:full}.
    The second equality follows by keeping only the main term of \(\psi_{r}(m,y,1)\).
\end{proof}

The next lemma computes \(\e{\recnyr}\) when the tree is not full. The reason why it is formulated in
terms of \(\mbar\) will be clear in the proof of \autoref{lem:mbar}.

\begin{lemma}\label{lem:expectation}
    Let \(\varphi_{r}(n,y)\eqd\e \recnyR{r}\).
    Let
    \begin{equation}\label{psi:bar}
        \begin{aligned}
            \psibar_{r}(n,m,z)
            &
            \eqd
            2^{m+1}
            \psi_{r}(m,z,2)
            -
            (2^{m+1}-n)
            \psi_{r}(m,z,1)
            \\
            &
            =
            n \psi_{r}(m,z,1)
            +
            \frac{\left({k!}\right)^{\frac{r}{k}}}{k{r!}}
            \frac{2^{m+1}}{m^{1+\frac{r}{k}} }
            \gamma
            \left(1+\frac{r}{k},\frac{m z^k}{k!}\right)
            .
        \end{aligned}
    \end{equation}
    If \(2^{\mbar} - 1 \le n \le 2^{\mbar+1}-1\), then
    \begin{equation}\label{eq:phi}
            \varphi_{r}(n,y)
            =
            \psibar_{r}(n,\mbar,y)
            +
            \bigO{
                n
                \mbar^{-\frac{1+r}{k}-1}
            }
            .
    \end{equation}
\end{lemma}

\begin{proof}
    Assume first that \(\mbar = m\).  When the tree is not necessarily full, the estimate of
    \(\varphi_{r}(n,y)\) in \eqref{eq:y:full} over counts the number of nodes at height \(m\) by
    \(2^{m+1}-n\). The contribution of the over counted nodes in \eqref{eq:y:full} can be estimated
    using \eqref{eq:y:last}. Subtracting this part from \eqref{eq:y:full} gives \eqref{eq:phi}.

    The only other possible case is that \(\mbar=m+1\) and the tree is full.  The result follows
    easily by adding an extra node \(v\) at height \(\mbar\), computing the total expectation of
    \(r\)-records for this tree by the case already studied, and subtracting
    \(\E{\recordRV{r}{v}|\clockRV{k}{o}=y}\sim\psi_{r}(m,y,1)\) from \eqref{eq:y:last}.
\end{proof}

\begin{corollary}
    We have
    \begin{equation}\label{eq:mean}
        \begin{aligned}
            &
            \e \recnR{r}
            =
            \rescaleInv{r}
            (\mu_{r,n}-\lg\lg n)
            +
            \frac{C_{2}(r) 2^{m+1}}{\lg(n)^{\frac{r}{k}+1}}
            +
            \bigO{
                n \lg(n)^{-\frac{r+1}{k}-1}
            }
            ,
        \end{aligned}
    \end{equation}
    where \(\mu_{r,n}\) is defined in \eqref{eq:mu:new}.
\end{corollary}

\begin{proof}
\autoref{lem:expectation} gives an asymptotic
expansion of \(\varphi_{r}(n,y) \eqd \E{\recnR{r}|\clockRV{k}{o}=y}\). 
To get rid of this conditioning, first consider a full binary tree of height \(m'=m+1\), i.e., a
tree of size \(n'=2^{m+2}-1\). It is easy to see that \(\varphi_{r}(n',\infty)\) is exactly twice of
\(\e \recnR{r}\) for \(n=2^{m+1}-1\). This solves the case when the tree is full.

The general case can be solved similarly. Consider a binary tree, with the right subtree of the root
being \(\Tbinn\) (possibly not full), and the left subtree of the root being \(\Tbin_{2^{m+1}-1}\),
i.e., a full binary tree of height \(m\). This tree has size \(n''=n+2^{m+1}\). Thus
\(\varphi_{r}(n'',\infty)\) is the expected number of \(r\)-records in \(\Tbinn\), plus the expected
number of \(r\)-record in  \(\Tbin_{2^{m+1}-1}\), which is \(\varphi(n',\infty)/2\) by the
previous paragraph. Thus
\begin{equation}\label{eq:mean:phi}
    \e \recnR{r} =
    \varphi_{r}(n'',\infty)-\frac{1}{2} \varphi_{r}(n',\infty),
\end{equation}
which implies \eqref{eq:mean} by \autoref{lem:expectation}.
\end{proof}

\begin{myRemark}
    Comparing \eqref{eq:mean} and \eqref{eq:rec:1:dist} in \autoref{thm:1:rec}, we see that
    \(\recnr\) is concentrated well above their means (at a distance of about \(n
        \lg(\lg(n))/\lg(n)^{1+r/k}\)). Thus \(\p{\recnr < \e \recnr} \to 0\). See also Remark 1.4 of
    \cite{janson04}.
\end{myRemark}

\begin{myRemark}
    The simplest case that \(r=k\) and the tree is full can also be computed directly by noticing
    that
    \begin{equation}\label{eq:mean:k:1}
        \begin{aligned}
        \e{\recnR{k}} 
            &
        =
        \sum_{v} \frac{1}{h(v)+1}
        =
        \sum_{i=0}^{m} \frac{2^{i} }{i+1}
        =
        -2^{m+1} \Phi (2,1,m+2)-\frac{1}{2} (\ii \pi )
        \\
        &
        =
        \frac{2^{m+1}}{m+2}\left(1+\sum _{n=1}^{N-1} (-1)^{n-1} (m+2)^{-n}
            \text{Li}_{-n}(2)+\bigO{m^{-N}}\right)
        &
        (N \in \dsN)
        \\
        &
        =2^{m+1}
        \left({m}+{2}{m^{-3}}+{6}{m^{-4}}+{38}{m^{-5}}+O\left({m^{-6}}\right)\right)
        &
        (N = 5)
        ,
        \end{aligned}
    \end{equation}
    where
    \(\Phi (z,s,a)\) denotes Hurwitz-Lerch zeta function
    \cite[25.14]{DLMF}, \(\text{Li}_{s}(z)\) denotes the polylogarithm function \cite[25.12]{DLMF},
    and the last step uses an asymptotic expansion of \(\Phi (z,s,a)\) given
    in \cite{Cai011}.
\end{myRemark}

\begin{lemma}\label{lem:var}
    We have
    \begin{equation}\label{eq:var}
        \V{\recnyr} = \bigO{n^2 m^{-\frac{2 r+1}{k}}}.
    \end{equation}
\end{lemma}

\begin{proof}
    Consider two nodes, \(v\) and \(w\) of heights \(s\) and \(t\) respectively. Let \(u\) be the 
    node that is furthest away from the root among the common ancestor of \(v\) and \(w\). Let \(i=h(u)\).

    We call the pair \( (v,w)\) \emph{good} if \(i \le \frac{m}{3}\) and \(s, t \ge \frac{2 m}{3}\).
    Otherwise we call it \emph{bad}.
    Assume for now that \( (v,w)\) is good.

    Let \(o = u_{0},\dots,u_{i}=u\) be the path from the root to \(u\). Let \(Z = \min_{1 \le j \le
    i} \clockRV{k}{u_{i}}\). 

    Note that conditioning on \(\clockRV{k}{o}=y\) and \(Z=z\), the events that
    \(v\) is an \(r\)-record and that \(w\) is an \(r\)-record are independent. 
    Thus by \autoref{lem:expectation:m} and the assumption that \( (v,w)\) is good,
    \begin{equation}
            \E{\recordRV{r}{v} \recordRV{r}{w}|\clockRV{k}{o}=y,Z=z}
            =
            \psi_{r}^{*}(s-i,z \wedge y)
            \psi_{r}^{*}(t-i,z \wedge y)
            +
            \bigO{
                m^{-\frac{2 r+1}{k}}
            }
            ,
        \label{eq:var:0}
    \end{equation}
    where \(a \wedge b \eqd \min\{a,b\}\).

    Since \(\psi_{r}^{*}(a,w)\) is increasing in \(w\), \eqref{eq:var:0} implies that, after
    averaging over \(z\),
    \begin{equation}\label{eq:var:1}
        \E{\recordRV{r}{v} \recordRV{r}{w}|\clockRV{k}{o}=y}
        \le
        \psi_{r}^{*}(s-i,y)
        \psi_{r}^{*}(t-i,y)
        +
        \bigO{
            m^{-\frac{2 r+1}{k}}
        }
        .
    \end{equation}
    On the other hand, again by \autoref{lem:expectation:m} and the assumption that \( (v,w)\) is
    good,
    \begin{equation}\label{eq:var:2}
        \begin{aligned}
            \E{\recordRV{r}{v}|\clockRV{k}{o}=y}
            \E{\recordRV{r}{w}|\clockRV{k}{o}=y}
            =
            \psi_{r}^{*}(s,y)
            \psi_{r}^{*}(t,y)
            +
            \bigO{
                m^{-\frac{2 r+1}{k}}
            }
            .
        \end{aligned}
    \end{equation}
    Therefore, by the definition of \(\psi_{r}^{*}(a,w)\) in \eqref{eq:psi:main}, 
    the first order term of the above is 
    \begin{equation}\label{eq:var:3}
        \begin{aligned}
            \Cov{\recordRV{r}{v},\recordRV{r}{w}|\clockRV{k}{o}=y}
            &
            \le
            \psi_{r}^{*}(s-i,y)
            \psi_{r}^{*}(t-i,y)
            -
            \psi_{r}^{*}(s,y)
            \psi_{r}^{*}(s,y)
            +
            \bigO{
                m^{-\frac{2 r+1}{k}}
            }
            \\
            &
            =
            \bigO{m^{-\frac{2r}{k}} }
            \left(
                i m^{-1}
                +
                \Gamma \left(\frac{r}{k},\frac{s y^k}{\Gamma (k+1)}\right)
                -
                \Gamma \left(\frac{r}{k},\frac{(s-i) y^k}{\Gamma (k+1)}\right)
            \right.
            \\
            &
            \qquad
            +
            \Gamma \left(\frac{r}{k},\frac{t y^k}{\Gamma (k+1)}\right)
            -
            \Gamma \left(\frac{r}{k},\frac{(t-i) y^k}{\Gamma (k+1)}\right)
            \\
            &
            \qquad
            +
            \Gamma \left(\frac{r}{k},\frac{(s-i) y^k}{\Gamma (k+1)}\right) 
            \Gamma \left(\frac{r}{k},\frac{(t-i) y^k}{\Gamma (k+1)}\right)
            \\
            &
            \qquad
            \left.
                -
                \Gamma \left(\frac{r}{k},\frac{s y^k}{\Gamma (k+1)}\right) 
                \Gamma \left(\frac{r}{k},\frac{t y^k}{\Gamma (k+1)}\right)
            \right)
        \end{aligned}
        .
    \end{equation}
    For \(x_{1}\le x_{2}\) and \(0 \le a \le 1\),
    \begin{equation}\label{eq:gam:diff}
        \Gamma (a,x_{1})-\Gamma (a,x_{2})
        =
        \int_{x_{1}}^{x_{2}} \ee^{-x} x^{a-1} \, \dd x
        \le
        \ee^{-x_{1}} x_{1}^{a-1} (x_{2}-x_{1})
        \le
        \left( 
            \frac{a}{\ee}
        \right)^{a}
        \frac{x_{2}-x_{1}}{x_{1}}
        ,
    \end{equation}
    since \(\ee^{-x} x^{a-1}\) is decreasing and \(\ee^{-x} x^a\leq \left(\frac{a}{\ee}\right)^a \). Thus
    when \((v,w)\) is good,
    \begin{equation}\label{eq:gam:diff:1}
        \left|
            \Gamma \left(\frac{r}{k},\frac{s y^k}{\Gamma (k+1)}\right)
            -
            \Gamma \left(\frac{r}{k},\frac{(s-i) y^k}{\Gamma (k+1)}\right)
        \right|
        =
        \bigO{\frac{i}{m}}
        .
    \end{equation}
    Cancelling other terms in \eqref{eq:var:3} in a similar way shows that
    \begin{equation}\label{}
        \Cov{\recordRV{r}{v},\recordRV{r}{w}|\clockRV{k}{o}=y}
        =
        \bigO{m^{-\frac{2 r+1}{k}}+i m^{-1-\frac{2r}{k}}}
        .
    \end{equation}
    Given \(i,s,t\), there are at most \(2^{s+t-i}\) choices of \(u,v,w\). Thus
    \begin{equation}\label{eq:var:sum:good}
        \begin{aligned}
            &
            \sum_{\text{good }(v,w)}
            \Cov{\recordRV{r}{v},\recordRV{r}{w}|\clockRV{k}{o}=y}
            \\
            &
            \qquad
            \le
            \sum_{i=1}^{m}
            \sum_{s=1}^{m}
            \sum_{t=1}^{m}
            2^{s+t-i}
            \bigO{i m^{-1-\frac{2r}{k}}+m^{-\frac{2 r+1}{k}}}
            =
            \bigO{n^{2} m^{-\frac{2 r+1}{k}}}
            .
        \end{aligned}
    \end{equation}
    The number of bad pairs is at most
    \begin{equation}\label{eq:var:sum:bad}
        \sum_{i>\frac{m}{3},s,t \le m}
        2^{s+t-i}
        +
        2
        \sum_{i>0,t < \frac{2 m}{3}, s\le m}
        2^{s+t-i}
        =
        \bigO{2^{2 m-\frac{m}{3}}}
        =
        \bigO{n^{\frac{5}{3}}}
        .
    \end{equation}
    Using the fact that \( \Cov{\recordRV{r}{v},\recordRV{r}{w}|\clockRV{k}{o}=y} \le 1\), it
    follows from \eqref{eq:var:sum:good} and \eqref{eq:var:sum:bad} that
    \begin{equation}\label{eq:var:final}
        \V{\recnyr}
        =
        \sum_{v,w}
        \Cov{\recordRV{r}{v},\recordRV{r}{w}|\clockRV{k}{o}=y}
        =
        \bigO{n^{2} m^{-\frac{2 r+1}{k}}}
        ,
    \end{equation}
    as the lemma claims.
\end{proof}

Recall that \(L\eqd\floor*{\left(2-\frac{1}{2 k}\right)\lg\lg n}\). Let \( \mySeqBig{v_{i}}{1 \le i \le 2^{L}}\) be the
\(2^{L}\) nodes of height \(L\). Let \(Y_{i}\) be the minimum of the \(\clockRV{k}{v}\) for all
nodes \(v\) on
the path between the root and \(v_{i}\).

\begin{lemma}\label{lem:sum:psi}
    We have
    \begin{equation}\label{eq:sum:psi}
        \recnR{r}=
        \sum _{i=1}^{2^L} \varphi_{r}(n_{i},Y_{i})
        +
        O_{p}\left(n m^{-1-\frac{1}{4 k} - \frac{r}{k}}\right)
        .
    \end{equation}
\end{lemma}

\begin{proof}
    The proof uses the estimate of the variance in \autoref{lem:var} and exactly the same argument
    of Lemma 2.3 in \cite{janson04}. We omit the details.
\end{proof}

\section{Transformation into a triangular array}

In this section, we prove \autoref{lem:rescale}, which shows that \(\recnR{r}\), properly rescaled and shifted, can 
be written as a sum of independent random variables. Three technical lemmas
\autoref{lem:Y1:bound}, \autoref{lem:mbar}, \autoref{lem:triangular} are needed.

\begin{proposition}\label{lem:rescale}
    Let \(\alpha_{n}\eqd\fracPart{\lg n}\) and 
    \(\beta_{n}\eqd\fracPart{\lg \lg n}\).
    Then
    \begin{equation}\label{eq:rescale}
        \begin{aligned}
            &
            \frac{m^{\frac{r}{k}+1}}{n C_{2}(r)}\recnR{r}
            -
            \frac{k}{r} \lg (n)
            -\sum _{i=1}^k 
            C_{1}(r,i)
            \lg (n)^{1-\frac{i}{k}} 
            -\lg (\lg (n))
            \\
            &
            \qquad
            \qquad
            =
            2^{1-\alpha_{n} }+\alpha_{n} -\beta_{n}
            -\ell + L + 1
            -C_{3}(r) \sum_{v:h(v)\leq L} \xi_{r,v}
            +\op{1}
            ,
        \end{aligned}
    \end{equation}
    where
    \begin{equation}\label{eq:xi:r:v}
        \xi_{r,v}
        \eqd
        \frac{m n_{v}}{n}\Gamma \left(\frac{r}{k},\frac{m \clockRV{k}{v}^k}{k!}\right)
        ,
    \end{equation}
    and
    \begin{equation}\label{eq:C123}
        \begin{aligned}
            &
            C_{1}(r,i)
            \eqd
            C_{7}(r,i)
            +
            \sum_{j=1}^{i}
            C_{8}(r,j,j k + i)
            ,
            \\
            &
            C_{2}(r)
            \eqd
            \frac{(k!)^{r/k} \Gamma \left(\frac{r}{k}\right)}{k^2 \Gamma (r)}
            ,
            \quad
            C_{3}(r)
            \eqd
            \frac{1}{\Gamma \left(1+\frac{r}{k}\right)}
            ,
            \\
            &
            C_{7}(r,i)
            \eqd
            \frac{(-1)^i k (k!)^{i/k} \Gamma \left(\frac{i+r}{k}\right)}{r i! \Gamma \left(\frac{r}{k}\right)}
            ,
            \quad
            C_{8}(r,j,b)
            \eqd
            \frac{k (k!)^{b/k} C_{6}(j,b) \Gamma \left(\frac{b+r}{k}\right)}{r \Gamma \left(\frac{r}{k}\right)}
            .
        \end{aligned}
    \end{equation}
\end{proposition}

\begin{proof}[Proof of {\autoref{lem:rescale}}]
    Expanding \eqref{eq:sum:tri} in \autoref{lem:triangular} bellow
    and dividing both sides by \(n m^{-\frac{r}{k}-1} C_{2}(r)\) shows that
    \begin{equation}\label{eq:rescale:1}
        \begin{aligned}
            \frac{m^{\frac{r}{k}+1}}{n C_{2}(r)}\recnR{r}
            =
            &
            \frac{k m}{r}+L+\frac{2^{m+1}}{n}+1
            +
            \underset{i=1}{\overset{k}{\sum }} C_{7}(r,i) m^{1-\frac{i}{k}}
            \\
            &
            +
            \underset{j=1}{\overset{k}{\sum }}\underset{b=j (k+1)}{\overset{(j+1) k}{\sum }} C_{8}(r,j,b) m^{-\frac{b}{k}+j+1}
            -
            C_{3}(r) \underset{v}{\sum }\xi_{r,v}
            +
            \bigO{m^{-\frac{1}{4 k}}}
            .
        \end{aligned}
    \end{equation}
    Subtracting 
    \begin{equation}\label{eq:rescale:sub}
            m^{\frac{r}{k}+1} \lg (n)^{-\frac{r}{k}-1} 
            \left( 
            \frac{k}{r} \lg (n) +\sum _{i=1}^k C_{1}(r,i) \lg (n)^{1-\frac{i}{k}} +\lg (\lg (n))
            \right)
            ,
    \end{equation}
    from both sides of \eqref{eq:rescale:1} gives \eqref{eq:rescale}.
\end{proof}

\begin{lemma}\label{lem:Y1:bound}
    Recall that \(Y_{1}\) has the distribution of the minimum of \(L+1\) independent \(\Gam(k,1)\)
    random variables. Let \(\mpr \eqd m-L\). 
    Let \(a>0\) be a constant.
    Then
    \begin{align}
        &
        \mathbb{E}\left[\Gamma \left(a,\frac{\mpr Y_{1}^k}{k!}\right)\right]
        =
        \bigO{\frac{L}{m}}
        &
        \text{if $a > 0$}
        ,
        \label{eq:Y1:U3}
        \\
        &
        \mathbb{E}\left[\Gamma \left(a,\frac{\mpr Y_{1}^k}{k!},\frac{m Y_{1}^k}{k!}\right)\right]
        =
        \bigO{
            \frac{L^{2}}{m^{2}}
        }
        &
        \text{if $1 > a > 0$}
        \label{eq:Y1:U2}
        ,
        \\
        &
        \E{
            \mpr^{-a} \Gamma \left(a,\frac{\mpr Y_{1}^k}{k!}\right)-m^{-a} \Gamma \left(a,\frac{m Y_{1}^k}{k!}\right)
        }
        =
        \bigO{
            \frac{
                L^2
            }{
                m^{a+2}
            }
        }
        &
        \text{if $1 > a > 0$}
        \label{eq:Y1:diff}
        .
    \end{align}
\end{lemma}

\begin{proof}
    Since
    \begin{equation}\label{eq:Y1:cdf}
        \p{Y_{1}>x} = \p{\Gam(k,1)>x}^{L+1} =Q(k,x)^{L+1}
        ,
    \end{equation}
    the density of \(Y_{1}\) is
    \begin{equation}\label{eq:Y1:pdf}
        g_{Y_{1}}(x)
        =
        \begin{cases}
            \dfrac{\left(1+L\right)}{\Gamma (k)}\ee^{-x} x^{k-1} Q(k,x)^{L} & \text{if $x \ge 0$},
            \\
            0 & \text{if $x < 0$},
        \end{cases}
    \end{equation}
    by the derivative formula
    \begin{equation}\label{eq:Gam:D}
        \frac{\mathrm{d}}{%
            \mathrm{d}z}Q\left(a,z\right)
        =
        -
        \frac{
        z^{a-1}\ee^{-z}
        }{
            \Gamma\left( a \right)
        }
        ,
        \quad
        \frac{\dd }{\dd x}Q^{-1}(a,x)=-\Gamma (a) \exp\left({Q^{-1}(a,x)}\right) Q^{-1}(a,x)^{1-a}
        ,
    \end{equation}
    see  \cite[8.8.13]{DLMF}.
    For \(0 < a \le 1\) and \(z \ge 0\), by the inequality \cite[8.10.11]{DLMF},
    \begin{equation}\label{eq:gam:ieq}
        \Gamma\left( a, z \right) 
        \le 
        \Gamma(a) (1 - (1 - \ee^{-z})^a)
        \le
        \Gamma(a) \ee^{-z}
        .
    \end{equation}
    Therefore,
    \begin{equation}\label{eq:U3:0}
        \begin{aligned}
        \mathbb{E}\left[\Gamma \left(a,\frac{\mpr Y_{1}^k}{k!}\right)\right]
        &
        =
        \int_{0}^{\infty}
        g_{Y_{1}}(x)
        \Gamma \left(a,\frac{\mpr x^k}{k!}\right)
        \,
        \dd x
        \\
        &
        \le
        \bigO{L}
        \int_{0}^{\infty}
        x^{k-1}
        \exp \left(-\frac{\mpr x^k}{k!}\right)
        \dd x
        =
        \bigO{\frac{L}{m}}
        .
        \end{aligned}
    \end{equation}
    For \(a > 1\) and \(z \ge 0\), also by \cite[8.10.11]{DLMF},
    \begin{equation}\label{eq:gam:ieq:1}
        \begin{aligned}
            \Gamma\left( a, z \right) 
            &
            \le
            \Gamma (a) 
            \left(1-\left(1-\exp\left({-\frac{m \Gamma (a+1)^{-1/a} x^k}{k!}}  \right)\right)^a\right)
            \\
            &
            \le
            a 
            \Gamma (a) 
            \exp\left({-\frac{m \Gamma (a+1)^{-1/a} x^k}{k!}}\right)
            ,
        \end{aligned}
    \end{equation}
    where the last inequality follows from that \( (1-b)^{a} \ge 1- a b\) for \(b \in (0,1)\) and
    \(a > 1\).
    Therefore, similar to \eqref{eq:U3:0},
    \begin{equation}\label{eq:U3:1}
        \begin{aligned}
        \mathbb{E}\left[\Gamma \left(a,\frac{\mpr Y_{1}^k}{k!}\right)\right]
        &
        \le
        \bigO{L}
        \int_{0}^{\infty}
        x^{k-1}
        \exp\left( 
        {-\frac{m x^k}{k!}\Gamma (a+1)^{-\frac{1}{a}}}
        \right)
        \,
        \dd x
        =
        \bigO{\frac{L}{m}}
        .
        \end{aligned}
    \end{equation}
    Thus we have \eqref{eq:Y1:U3}.

    For \eqref{eq:Y1:U2}, first by \eqref{eq:Y1:pdf}, 
    \begin{equation}\label{eq:U2:0}
        \mathbb{E}\left[\Gamma \left(a,\frac{\mpr Y_{1}^k}{k!},\frac{m Y_{1}^k}{k!}\right)\right]
        =
        \int_{0}^{\infty}
        g_{Y_{1}}(x)
        \Gamma \left(a,\frac{\mpr x^k}{k!}, \frac{m x^{k}}{k!}\right)
        \,
        \dd x
        .
    \end{equation}
    Since \(\ee^{-x} x^{a-1}\) is decreasing when \(0 < a\le 1\), for \(0 < x_{1} < x_{2}\)
    \begin{equation}\label{eq:gamma:x1:x2}
        \Gamma\left( a, x_{1}, x_{2} \right)
        =
        \int_{x_{1}}^{x_{2}} \ee^{-x} x^{a-1} \, \dd x
        \le
        (x_{2}-x_{1})\ee^{-x_{1}} x_{1}^{a-1}
        .
    \end{equation}
    Therefore,
    \begin{equation}\label{eq:gamma:x1:x2:1}
        \Gamma \left(a,\frac{\mpr x^k}{k!},\frac{m x^k}{k!}\right)
        \le
        L \hat{m}^{a-1} (k!)^{-a} x^{a k} \exp\left({-\frac{\hat{m} x^k}{k!}}\right)
        .
    \end{equation}
    Substituting the above inequality into \eqref{eq:U2:0} and integrating gives \eqref{eq:Y1:U2}.

    For \eqref{eq:Y1:diff}, note that
    \begin{equation}\label{eq:Y1:diff:0}
        \qquad
        \begin{aligned}
            &
            \hat{m}^{-a} \Gamma \left(a,\frac{Y_{1}^k \hat{m}}{k!}\right)-m^{-a} \Gamma \left(a,\frac{m Y_{1}^k}{k!}\right)
            \\
            &
            =
            \left(\hat{m}^{-a}-m^{-a}\right) \Gamma \left(a,\frac{\hat{m} Y_{1}^k}{k!}\right)+m^{-a}
            \Gamma \left(a,\frac{Y_{1}^k \hat{m}}{k!},\frac{m Y_{1}^k}{k!}\right)
            ,
        \end{aligned}
    \end{equation}
    where \(\Gamma(a,x_{0},x_{1})\eqd \Gamma(a,x_{0})-\Gamma(a,x_{1}).\)
    The result follows easily from \eqref{eq:Y1:U3} and \eqref{eq:Y1:U2}.
\end{proof}

The next two lemmas first remove the \(\mbar\) (see \autoref{lem:expectation}) hidden in the
representation \eqref{eq:sum:psi} then transform it into a sum of independent random variables.

\begin{lemma}\label{lem:mbar}
    Let \(n_{i}\) be the size of the subtree rooted at \(v_{i}\). Then
    \begin{equation}\label{eq:sum:Yi}
        \begin{aligned}
        \recnR{r}
        =
        &
        \psibar_{r}(n,m,\infty)
        +
        \frac{r (k!)^{r/k} \Gamma \left(\frac{r}{k}\right)}{k^2 r!} n m^{-\frac{k+r}{k}} L
        -
        \sum_{i=1}^{2^{L}}
        \frac{n_{i} }{k r!} 
        \left(\frac{m}{k!}\right)^{-\frac{r}{k}} 
        \Gamma \left(\frac{r}{k},\frac{m Y_{i} ^k}{k!}\right)
        \\
        &
        +
        O_{p}\left(
            n m^{-1-\frac{1}{4 k} - \frac{r}{k}}
        \right)
        .
        \end{aligned}
    \end{equation}
\end{lemma}

\begin{proof}
    By \autoref{lem:expectation}, we have
    \begin{equation}\label{eq:phi:ni}
        \begin{aligned}
        \varphi_{r}(n_{i},y)
        =
        &
        \frac{n_i (k!)^{r/k} \mpr^{-\frac{r}{k}} \gamma \left(\frac{r}{k},\frac{\mpr Y_{i}^k}{k!}\right)}{k r!}
        \\
        &
        +
        \frac{n_i (k!)^{r/k} \mpr^{-\frac{k+r}{k}} \gamma \left(\frac{k+r}{k},\frac{\mpr Y_{i}^k}{k!}\right)}{k r!}
        +
        \frac{2^{\mpr+1} (k!)^{r/k} \mpr^{-\frac{k+r}{k}} \gamma \left(\frac{k+r}{k},\frac{m Y_{i}^k}{k!}\right)}{k r!}
        \\
        &
        +
        \sum _{j=1}^k \sum _{b=j (k+1)}^{(j+1) k} 
        \frac{(k!)^{\frac{b+r}{k}}}{k r!}n_i C_{6}(j,b) \mpr^{j-\frac{b+r}{k}} \gamma \left(\frac{b+r}{k},\frac{\mpr Y_{i}^k}{k!}\right)
        \\
        &
        + 
        \sum _{i=1}^k \frac{(-1)^i n_i (k!)^{\frac{i+r}{k}} \mpr^{-\frac{i+r}{k}} \gamma \left(\frac{i+r}{k},\frac{\mpr Y_{i}^k}{k!}\right)}{k i! r!}
        +
        \Op{n_i \mpr^{-\frac{k+r+1}{k}}}
        ,
        \end{aligned}
    \end{equation}
    where \(\mpr=m-L=m-\bigO{\log m}\). (This is why we need to formulate \autoref{lem:expectation} in
    terms of \(\mbar\)--here \(\mpr\) is either the height of subtree rooted at \(v_{i}\), or it is
    the height of the subtree plus one and the subtree is full.)

    We now convert this into an expression in \(m\).
    Let
    \begin{equation}\label{eq:phi:ni:1}
        \begin{aligned}
        x_{i}
        =
        &
        \frac{n_i (k!)^{r/k} m^{-\frac{r}{k}} \Gamma \left(\frac{r}{k}\right)}{k r!}
        -
        \frac{n_i (k!)^{r/k} m^{-\frac{r}{k}} \Gamma \left(\frac{r}{k},\frac{m Y_{i}^k}{k!}\right)}{k r!}
        \\
        &
        +
        \frac{n_i (k!)^{r/k} m^{-\frac{k+r}{k}} \Gamma \left(\frac{k+r}{k}\right)}{k
            r!}
        +\frac{2^{\mpr+1} (k!)^{r/k} m^{-\frac{k+r}{k}} \Gamma \left(\frac{k+r}{k}\right)}{k r!}
        \\
        &
        +
        \sum _{j=1}^k \sum _{b=j (k+1)}^{(j+1) k} 
        \frac{(k!)^{\frac{b+r}{k}}}{k r!}n_i C_{6}(j,b) m^{j-\frac{b+r}{k}} \Gamma
            \left(\frac{b+r}{k}\right)
        \\
        &
        + 
        \sum _{i=1}^k \frac{(-1)^i n_i (k!)^{\frac{i+r}{k}} m^{-\frac{i+r}{k}} \Gamma
            \left(\frac{i+r}{k}\right)}{k i! r!}
        .
        \end{aligned}
    \end{equation}
    Then using the identity \(\gamma(a,z)=\Gamma(a)-\Gamma(a,z)\),
    \begin{align*}
        \varphi_{i}(n_{i},y)-x_{i}
        =
        &
        \frac{n_i (k!)^{r/k} 
            \left( 
                \mpr^{-\frac{r}{k}}
                -
                m^{-\frac{r}{k}}
            \right)
            \Gamma \left(\frac{r}{k}\right)}{k r!}
        \\
        &
        +
        \frac{n_i (k!)^{r/k} \left(m^{-\frac{r}{k}} \Gamma \left(\frac{r}{k},\frac{m Y_{i}^k}{k!}\right) 
                -\mpr^{-\frac{r}{k}} \Gamma \left(\frac{r}{k},\frac{\mpr Y_{i}^k}{k!}\right)
            \right) }{k r!}
        \\
        &
        +
        \frac{n_i (k!)^{r/k} 
            \left( 
                \mpr^{-\frac{k+r}{k}}
            -
            m^{-\frac{k+r}{k}}
            \right) \Gamma \left(\frac{k+r}{k}\right)
        }{k
            r!}
        \\
        &
        +
        \frac{n_i (k!)^{r/k} \mpr^{-\frac{k+r}{k}} \Gamma \left(\frac{k+r}{k},\frac{\mpr Y_{i}^k}{k!}\right)}{k r!}
        \\
        &
        +
        \frac{2^{\mpr+1} (k!)^{r/k} 
            \left( 
            \mpr^{-\frac{k+r}{k}} 
            -
            m^{-\frac{k+r}{k}} 
            \right)
            \Gamma \left(\frac{k+r}{k}\right)}{k r!}
        \\
        &
        +
        \frac{2^{\mpr+1} (k!)^{r/k} \mpr^{-\frac{k+r}{k}} \Gamma \left(\frac{k+r}{k},\frac{\mpr
                    Y_{i}^k}{k!}\right)}{k r!}\stepcounter{equation}\tag{\theequation}\label{eq:phi:ni:2}\\
        &
        +
        \sum _{j=1}^k \sum _{b=j (k+1)}^{(j+1) k} 
        \frac{(k!)^{\frac{b+r}{k}}}{k r!}n_i C_{6}(j,b) 
            \left( 
                \mpr^{j-\frac{b+r}{k}} 
                -
                m^{j-\frac{b+r}{k}} 
            \right)
            \Gamma
            \left(\frac{b+r}{k}\right)
        \\
        &
        +
        \sum _{j=1}^k \sum _{b=j (k+1)}^{(j+1) k} 
        \frac{(k!)^{\frac{b+r}{k}}}{k r!}n_i C_{6}(j,b) \mpr^{j-\frac{b+r}{k}} \Gamma \left(\frac{b+r}{k},\frac{\mpr Y_{i}^k}{k!}\right)
        \\
        &
        + 
        \sum _{i=1}^k \frac{(-1)^i n_i (k!)^{\frac{i+r}{k}} 
            \left( 
                \mpr^{-\frac{i+r}{k}} 
                -
            m^{-\frac{i+r}{k}} 
            \right)
            \Gamma
            \left(\frac{i+r}{k}\right)}{k i! r!}
        \\
        &
        + 
        \sum _{i=1}^k \frac{(-1)^i n_i (k!)^{\frac{i+r}{k}} \mpr^{-\frac{i+r}{k}} \Gamma \left(\frac{i+r}{k},\frac{\mpr Y_{i}^k}{k!}\right)}{k i! r!}
        +
        \Op{n_i m^{-\frac{k+r+1}{k}}}
        .
    \end{align*}

    The first term of the above expression is
    \begin{equation}\label{eq:phi:ni:3}
        \begin{aligned}
        \frac{
            n_i (k!)^{r/k} 
            \left( 
                \mpr^{-\frac{r}{k}}
                -
                m^{-\frac{r}{k}}
            \right)
            \Gamma
            \left(\frac{r}{k}\right)
        }{
            k
            r!
        }
        =
        \frac{
            r 
            (k!)^{r/k} 
            \Gamma \left(\frac{r}{k}\right)
            n_i m^{-\frac{r}{k}-1} L
        }{
            k^{2} r!
        }
        +
        \bigO{
            \frac{n_{i}L^{2}}{m^{\frac{r}{k}+2}}
        }
        ,
        \end{aligned}
    \end{equation}
    since
    \(\mpr^{-a} -m^{a} = a L m^{-a-1}+ \bigO{L^{2} m^{-a-2}}\). 
    The terms which do not contain \(Y_{i}\) can be bounded similarly.
    For terms involving \(Y_{i}\), we can use \autoref{lem:Y1:bound}.
    For example, by \eqref{eq:Y1:diff}, the second term is
    \begin{equation}\label{eq:phi:ni:4}
        \frac{n_i (k!)^{r/k} \left(m^{-\frac{r}{k}} \Gamma \left(\frac{r}{k},\frac{m Y_{i}^k}{k!}\right) 
                -\mpr^{-\frac{r}{k}} \Gamma \left(\frac{r}{k},\frac{\mpr Y_{i}^k}{k!}\right)
            \right) }{k r!}
        =
        \Op{
            n_{i} m^{-\frac{r}{k}-2} L^{2}
        }
        .
    \end{equation}
    In the end, it follows from \autoref{lem:Y1:bound} and simple asymptotic computations that
    \begin{equation}\label{eq:xi:diff}
        \phi_{r}(n_{i},y)-x_{i}
        =
        \frac{
            r 
            (k!)^{r/k} 
            \Gamma \left(\frac{r}{k}\right)
            n_i m^{-r/k-1} L
        }{
            k^{2} r!
        }
        +
        \Op{
            L^{2} n_{i} m^{-\frac{r+1}{k}-1}
        }
        .
    \end{equation}
    Since \(\sum_{i=1}^{2^{L}}n_{i} = n - (2^{L}-1)=n - \bigO{m^{2-\frac{1}{2 k}}}\),
    \begin{equation}\label{eq:xi:diff:sum}
        \sum_{i}^{2^{L}}
        \left( 
            \phi_{r}(n_{i},y)-x_{i}
        \right)
        =
        \frac{
            r 
            (k!)^{r/k} 
            \Gamma \left(\frac{r}{k}\right)
            n m^{-r/k-1} L
        }{
            k^{2} r!
        }
        +
        \Op{
            L^{2} n m^{-\frac{r+1}{k}-1}
        }
        .
    \end{equation}
    Thus by \eqref{eq:sum:psi}, we have
    \begin{equation}\label{eq:sum:psi:m}
        \begin{aligned}
        \recnR{r}
        =
        &
        \sum _{i=1}^{2^L} \varphi_{r}(n_{i},y)
        +
        O_{p}\left(n m^{-1-\frac{1}{4 k} - \frac{r}{k}}\right)
        \\
        =
        &
        \sum_{i}^{2^{L}} x_{i}
        +
        \frac{
            r 
            (k!)^{r/k} 
            \Gamma \left(\frac{r}{k}\right)
            n m^{-r/k-1} L
        }{
            k^{2} r!
        }
        +
        O_{p}\left(
            n m^{-1-\frac{1}{4 k} - \frac{r}{k}}
        \right)
        ,
        \end{aligned}
    \end{equation}
    from which \eqref{eq:sum:Yi} follows immediately.
\end{proof}

\begin{lemma}\label{lem:triangular}
    Let \(n_{v}\) be the size of the subtree rooted at the node \(v\). Then
    \begin{equation}\label{eq:sum:tri}
        \begin{aligned}
        \recnR{r}
        =
        &
        \psibar_{r}(n,m,\infty)
        +
        \frac{r (k!)^{r/k} \Gamma \left(\frac{r}{k}\right) }{k^2 r!}
        n m^{-\frac{k+r}{k}} L
        \\
        &
        -
        \sum_{v:h(v)\le L}
        \frac{n_{v} }{k r!} 
        \left(\frac{m}{k!}\right)^{-\frac{r}{k}} 
        \Gamma \left(\frac{r}{k},\frac{m \clockRV{k}{v} ^k}{k!}\right)
        +
        O_{p}\left(
            n m^{-1-\frac{1}{4 k} - \frac{r}{k}}
        \right)
            .
        \end{aligned}
    \end{equation}
\end{lemma}

\begin{proof}
    Recall that \(Y_{i}\) is the minimum of \(L+1\) independent \(\Gam(k,1)\) random variables
    \(\mySeqBig{\clockRV{k}{v}}{v \in P(v_{i})}\), where \(P(v_{i})\) denotes the path from the root
    \(o\) to \(v_{i}\).
    Let \(a=\left({2 k! \log (m)}/{m}\right)^{1/k}\).
    The probability that at least two
    \(\clockRV{k}{v}\) are less than \(a\) is
    \begin{equation}\label{eq:a}
        \begin{aligned}
        &
        1-\p{\Gam(k,1)>a}^{L+1}-L \p{\Gam(k,1)>a}^{L}\p{\Gam(k,1) \le a}
        \\
        &
        =
        1-Q(k,x)^{L+1}-L Q(k,x)^{L}(1-Q(k,x))
        \\
        &
        = \bigO{a^{2 k}L^{2}}
        = \bigO{\log(m)^{2} m^{-2} L^{2}}
        ,
        \end{aligned}
    \end{equation}
    where we use the approximation of \(Q(k,x)^{L}\) in \eqref{eq:fmexpand} and
    the series expansion of \(Q(k,x)\) in \cite[8.7.3]{DLMF}. Thus the probability that this happens
    for some \(i\) is \(\bigO{2^{L} \log(m)^{2} m^{-2} L^{2}}=o(1)\).

    With probability goes to \(1\), there is at most one \(\clockRV{k}{v}\) that is less than
    \(a\) on each path \(P(v_{i})\). When this happens, by the inequality \eqref{eq:gam:ieq},
    \begin{equation}\label{eq:Y:T}
        0 
        \le 
        \sum_{v \in P(v_{i})} 
        \Gamma\left( \frac{r}{k}, \frac{m \clockRV{k}{v}^{k}}{k!} \right)
        -
        \Gamma\left( \frac{r}{k}, \frac{m Y_{i}^{k}}{k!} \right)
        \le
        L 
        \Gamma\left( \frac{r}{k}, \frac{m a^{k}}{k!} \right)
        =
        \bigO{ m^{-2} L }
        .
    \end{equation}
    Therefore,
    \begin{equation}\label{eq:Y:T:1}
        \begin{aligned}
            \sum_{i=1}^{2^{L}}
            n_{i}
            \Gamma\left( \frac{r}{k}, \frac{m Y_{i}^{k}}{k!} \right)
            &
            =
            \sum_{i=1}^{2^{L}}
                n_{i}
                \sum_{v \in P(v_{i})} 
                \Gamma\left( \frac{r}{k}, \frac{m \clockRV{k}{v}^{k}}{k!} \right)
            +
            \bigO{n m^{-2} L }
            \\
            &
            =
            \sum_{h(v)\le L} 
            \Gamma\left( \frac{r}{k}, \frac{m \clockRV{k}{v}^{k}}{k!} \right)
                \sum_{i:v \in P(v_{i})}
                n_{i}
            +
            \bigO{n m^{-2} L }
            \\
            &
            =
            \sum_{h(v)\le L} 
            \Gamma\left( \frac{r}{k}, \frac{m \clockRV{k}{v}^{k}}{k!} \right)
            n_{v}
            +
            \bigO{n m^{-2} L }
            ,
        \end{aligned}
    \end{equation}
    where in the last step we use \(n_{v}-2^{L} \le \sum_{i:v \in P(v_{i})} n_{i} \le n\).
    Thus
    \begin{equation}\label{eq:Y:T:2}
        \sum_{i=1}^{2^{L}}
        \frac{n_{i} \left(\frac{m}{k!}\right)^{-\frac{r}{k}} }{k r!} 
        \Gamma \left(\frac{r}{k},\frac{m Y_{i} ^k}{k!}\right)
        =
        \sum_{h(v)\le L} 
        \frac{n_{v} \left(\frac{m}{k!}\right)^{-\frac{r}{k}} }{k r!} 
        \Gamma\left( \frac{r}{k}, \frac{m \clockRV{k}{v}^{k}}{k!} \right)
        +
        \bigO{
            n m^{-\frac{r}{k}-2} L
        }
        .
    \end{equation}
    The lemma follows by putting this into \eqref{eq:sum:Yi}.
\end{proof}

\section{Convergence of the triangular array}

By taking subsequences, we can assume that as \(n \to \infty\), \(\alpha_{n}\eqd\fracPart{\lg n}  \to \alpha\) and
\(\beta_{n}\eqd\fracPart{\lg \lg n} \to \beta\).
Thus \(\lg n = m + \alpha + o(1)\), \(\lg m = \lg \lg n + o(1)=l+\beta+o(1)\), where \(l \eqd
    \floor{\lg \lg n}\). Moreover,
\(\lg n - \lg \lg n = m - l + \alpha-\beta+o(1)\) and
\begin{equation}\label{eq:gamma}
    \fracPart{\lg n - \lg \lg n}
    \to
    \gamma =
    \begin{cases}
        \alpha - \beta & \text{if $\alpha > \beta$}, \\
        \alpha - \beta + 1 & \text{if $\alpha < \beta$}, \\
        \text{$0$ or $1$} & \text{if $\alpha = \beta$},
    \end{cases}
\end{equation}
which implies \(\gamma \equiv \alpha-\beta \pmod{1}\). 

\begin{lemma}\label{lem:convergence}
    Let \(h \eqd 2^{\beta -\alpha } \Gamma \left(\frac{r}{k}\right)\).
    Assume that \(\alpha_{n} \to \alpha\) and \(\beta_{n} \to \beta\).
    Then as \(n \to \infty\):
    \begin{enumerate}[(i)]
        \item\label{lem:con:null} 
            For all fixed \(x > 0\),
            \(
                \sup_{v} \p{\xi_{r,v}>x}\to 0
                .
            \)
        \item\label{lem:con:nu} 
            For all fixed \(x > 0\),
            \(
                \sum_{v:h(v)\le L} \p{\xi_{r,v}>x}\to \nurab(x,\infty),
            \)
            where \(\nurab\) is defined in \eqref{eq:nurab}.
        \item\label{lem:con:mean} 
            We have
            \begin{equation}\label{eq:con:nu}
                \begin{aligned}
                    &
                    \sum_{v:h(v)\le L} \E{\xi_{r,v} \mathbb 1 [\xi_{r,v}\le h]} 
                    -
                    \Gamma\left( 1+\frac{r}{k} \right)
                    \left( 
                        2^{1-\alpha }+\alpha -\beta -\ell + L
                    \right)
                    \\
                    &
                    \qquad
                    \to 
                    f_{r,k,\gamma}
                    -\int_{h}^{1} x 
                    \, \dd \nurab(x)
                    ,
                \end{aligned}
            \end{equation}
            where \(f_{r,k,\gamma}\) is a constant defined later in \eqref{eq:f:gamma}.
        \item\label{lem:con:var} 
            We have
            \begin{equation}\label{eq:con:sig}
            \sum_{v:h(v)\le L} \V{\xi_{r,v} \mathbb 1 [\xi_{r,v}\le h]} \to
                \int_{0}^{h} x^2 
                \, \dd \nurab(x)
                .
            \end{equation}
    \end{enumerate}
\end{lemma}

Before getting into the somewhat complicated proof of \autoref{lem:convergence}, we first show why
\autoref{thm:1:rec} and \autoref{thm:mean} follow from it.

Let \(\xi'_{i}\eqd \Gamma\left( 1+\frac{r}{k} \right) \left( 2^{1-\alpha }+\alpha -\beta -\ell + L
    \right)/n\), which are deterministic.  It follows from \autoref{lem:convergence} that we can
apply Theorem 15.28 in \cite{kallenberg02} with \(a=0\), \(b=f_{r,k,\gamma}\) to show that the
triangular array \(\sum_{h(v)\le L} \xirv + \sum_{i = 1}^{n} \xi_{i}'\) converges in distribution to
\(\Wrab\) (defined in \autoref{thm:1:rec}). Thus by \autoref{lem:rescale}, \autoref{thm:1:rec}
follows immediately.

For \autoref{thm:mean}, note that the right-hand-side of \eqref{eq:rec:r:dist} equals
\begin{equation}\label{eq:rec:r:dist:1}
    \begin{aligned}
        &
        \rescale{1}
        \left(
            \recnall
            -
            \sum_{r=1}^{k}
            \rescaleInv{r}
            \mu_{r,n}
        \right)
        \\
        &
        =
        \left(
            \rescale{1}
            \recni
            -
            \mu_{1,n}
        \right)
        +
        \sum_{r=2}^{k}
        \frac{
            C_{2}(r)
        }{
            C_{2}(1)
            \lg(n)^{\frac{r-1}{k}}
        }
        \left( 
        \rescale{r}
        \recnr
        -
        \mu_{r,n}
        \right)
        \\
        &
        =
        \left(
            \rescale{1}
            \recni
            -
            \mu_{1,n}
        \right)
        +
        o_{p}(1)
        \inlaw 1- C_{3}(1) W_{1,k,\gamma }
        ,
    \end{aligned}
\end{equation}
where the last two steps follow from \autoref{thm:1:rec}.

It remains to show that the result for \(\recnall\) holds for \(\recnalle\).  By identifying cutting an
edge with cutting its lower (closer to root) endpoint node, we see that \(\recnalle\) is distributed
as \(\recnall-k\) conditioned on \(\clockRV{k}{o}=\infty\).  Thus to apply the same argument for
\(\recnall\) to \(\recnalle\), we only need to make \(Y_{i}\) the minimum of \(L\), instead of
\(L+1\), independent \(\Gam(k,1)\) random variables, and to exclude \(\xi_{r,o}\), i.e., to exclude
the root, in the sum of the triangular array in \eqref{eq:rescale}.  This minor change certainly
does not matter.  See also the argument for \(k=1\) at the end of \cite{janson04}.

\subsection{Proof of {\autoref{lem:convergence}} {(\ref{lem:con:null})}}

Recall that in \eqref{eq:xi:r:v} we define
\begin{equation}\label{eq:xi:1:v}
    \xi_{r,v}
    \eqd
    \frac{m n_{v}}{n}\Gamma \left(\frac{r}{k},\frac{m \clockkv^k}{k!}\right)
    ,
\end{equation}
where \(\mySeqBig{\clockkv}{v \in \Tbinn}\) are \iid \(\Gam(k,1)\) random variables.
Thus \(\p{\clockRV{k}{v}>x}=Q(k,x)\), where \(Q(k,x) \eqd \Gamma(k,x)/\Gamma(k)\), see
    \eqref{eq:gam:density}.
Assume for now that \(\frac{n x}{\Gamma \left({1}/{k}\right) m n_{v}} \le 1\).
Then, for all fixed \(x > 0\).
\begin{equation}\label{eq:xi:null}
    \begin{aligned}
    \p{\xirv > x}
    &
    =
    \p{
        \Gamma \left(\frac{r}{k},\frac{m \clockkv^k}{k!}\right)>\frac{n x}{m n_{v}}
    }
    \\
    &
    =
    \p{
        \clockkv
        \le
        \left(\frac{k! }{m}Q^{-1}\left(\frac{r}{k},\frac{n x}{\Gamma \left(\frac{r}{k}\right) m
                    n_{v}}\right)\right)^{\frac{1}{k}}
    }
    \\
    &
    =
    1-Q\left(k,\left(\frac{k!}{m} Q^{-1}\left(\frac{r}{k},\frac{n x}{\Gamma \left(\frac{r}{k}\right)
                    m  n_{v}}\right)\right)^{1/k}\right)
    .
    \end{aligned}
\end{equation}

The function \(Q^{-1}(a,z)\) is only defined for \(z \in (0,1]\).  However, we can extend its domain to
\( (0,\infty)\) by letting \(Q^{-1}(a,z)=0\) for \(z > 1\).  This extension makes \eqref{eq:xi:null}
also valid for \(\frac{n x}{\Gamma \left({1}/{k}\right) m n_{v}} > 1\), since in this case every
expression in \eqref{eq:xi:null} equals \(0\).

By \cite[8.10.11]{DLMF}, for \(z \ge 0\),
\begin{equation}\label{eq:ieq:gam}
    1-\left(1-\exp\left(-\Gamma\left( 1+\frac{r}{k}
            \right)^{-{k}/{r}}z\right)\right)^{{r}/{k}}
    \le
    Q\left(\frac{r}{k},z\right)
    \leq 
    1-\left(1-\ee^{-z}\right)^{r/k}
    .
\end{equation}
Letting \(y=Q\left( \frac{r}{k},z \right) \in (0,1]\), \eqref{eq:ieq:gam} implies that
\begin{equation}\label{eq:ieq:gam:1}
    Q^{-1}\left(\frac{r}{k},y\right)
    =
    z
    \leq \log_{+} \left(\frac{1}{1-(1-y)^{k/r}}\right)
    \le
    \log_{+}\left( \frac{1}{y} \right)
    ,
\end{equation}
where \(\log_{+}(z)\eqd \max\{\log(z), 0\}\).
Similarly, it follows from \eqref{eq:ieq:gam} that
\begin{equation}\label{eq:ieq:gam:2}
    Q^{-1}\left(\frac{r}{k},y\right)
    \ge
    \Gamma \left(1+\frac{r}{k}\right)^{k/r} \log \left(\frac{r}{k y}\right)
    .
\end{equation}
Note that \eqref{eq:ieq:gam:1} and \eqref{eq:ieq:gam:2} also hold for \(y > 1\) by our extension of \(Q^{-1}(a,z)\).
Thus uniformly for all \(v\) with \(h(v)\le L\),
\begin{equation}\label{eq:third:para}
    \begin{aligned}
        \frac{k!}{m} Q^{-1}\left(\frac{r}{k},\frac{n x}{\Gamma \left(\frac{r}{k}\right) m
        n_{v}}\right)
        \le
        \frac{k!}{m}
        \log_{+}\left(\frac{\Gamma \left(\frac{r}{k}\right) m n_{v}}{n x}\right)
        =
        \bigO
        {
            \frac{\log_{+} \frac{m}{x}}{m}
        }
    ,
    \end{aligned}
\end{equation}
where the last step uses that \(n_{v} \le n\).
Thus we can apply the series expansion of \(Q(k,z)\) near \(z=0\) in \cite[8.7.3]{DLMF} to
\eqref{eq:xi:null} to get
\begin{equation}\label{eq:prob:k}
    \begin{aligned}
        \p{\xirv > x} 
        &
        =
        \frac{1}{m}Q^{-1}\left(\frac{r}{k},\frac{n x}{\Gamma \left(\frac{r}{k}\right) m n_{v}}\right)
        \left( 
            1
            +
            \bigO
            {
                \frac{1}{m}Q^{-1}\left(\frac{r}{k},\frac{n x}{\Gamma \left(\frac{r}{k}\right) m n_{v}}\right)
            }^{\frac{1}{k}}
        \right)
        \\
        &
        =
        \frac{1}{m}Q^{-1}\left(\frac{r}{k},\frac{n x}{\Gamma \left(\frac{r}{k}\right) m n_{v}}\right)
        \left( 
            1
            +
            \bigO
            {
                \frac{\log_{+} \frac{m}{x}}{m}
            }^{\frac{1}{k}}
        \right)
        =
        \bigO{
            \frac{\log_{+} \frac{m}{x}}{m}
        }
        .
    \end{aligned}
\end{equation}
Therefore this probability tends to zero for all fixed \(x\).

\subsection{Proof of {\autoref{lem:convergence}} {(\ref{lem:con:nu})}}

We reuse the notion of \emph{good} and \emph{bad} nodes defined in \cite[pp.~250]{janson04}. A good
node \(v\) has \(n_v=2^{m-t}-1\) for some \(t\) with \(l/2 \le t \le L\). All other nodes with
height at most \(L\) are called \emph{bad}. \citet[eq.~20]{janson04} showed that
\begin{equation}\label{eq:good:bad}
    \# \left\{v \text{ good}:n_{v}=2^{m-t}-1 \right\}
        =
        \begin{cases}
            2^{t+\alpha_{n}} + \bigO{1}
            &
            \frac{l}{2} \le t < L,
            \\
            (2-2^{\alpha_{n}}) 2^{L} + \bigO{1}
            &
            t = L
            ,
        \end{cases}
\end{equation}
and that the number of bad nodes is \(\bigO{L+2^{l/2}}=\bigO{m^{1/2}}\).

As we have shown in \eqref{eq:prob:k} that \(\p{\xirv>x}=\bigO{m^{-1}\log_{+}\frac{m}{x}}\). By the same argument as in \cite[Eq.~21, 22]{janson04}, the bad nodes can be
ignored in the proof of (\ref{lem:con:nu}), (\ref{lem:con:mean}) and (\ref{lem:con:var}) of
\autoref{lem:convergence}.

Note that for \(t \ge L\), \(m 2^{m-t} \le 2^{l+1+m-L} < \frac{n x}{\Gamma\left( {r}/{k} \right)}\)
for \(n\) large enough, which implies \(Q^{-1}\left(\frac{1}{k},\frac{n x}{\Gamma
\left({r}/{k}\right) m n_{v}}\right)=0\) by our extension of \(Q^{-1}(a,z)\).
Thus, it follows from \eqref{eq:prob:k} and \eqref{eq:good:bad} that
    \begin{align*}\label{eq:sum:probk}
    \sum_{v \text{ good}}
    \p{\xirv > x}
    &
    \sim 
    \sum_{t \ge l/2} (2^{t+\alpha_{n}}+\bigO{1})
    \frac{1}{m}Q^{-1}\left(\frac{r}{k},\frac{n x}{\Gamma \left(\frac{r}{k}\right) m n_{v}}\right)
    +
    o(1)
    \\
    &
    =
    \sum_{t \ge l/2} 2^{t+\alpha_{n}-l-\beta_{n}}
    Q^{-1}\left(\frac{r}{k},2^{t-l+\alpha_{n}-\beta_{n}+o(1)}
        \frac{x}{\Gamma\left(\frac{r}{k}\right)}\right)
    +
    o(1)
    \\
    &
    =
    \sum_{i \le l/2} 2^{-i+\alpha-\beta}
    Q^{-1}\left(\frac{r}{k},2^{-i+\alpha-\beta+o(1)}
        \frac{x}{\Gamma\left(\frac{r}{k}\right)}\right)
    +
    o(1)\numberthis{}
    \\
    &
    \to
    F_{r}(x)
    \eqd
    \sum_{-\infty}^{\infty} 
    2^{-i+\alpha-\beta}
    Q^{-1}\left(\frac{r}{k},2^{-i+\alpha-\beta}
        \frac{x}{\Gamma\left(\frac{r}{k}\right)}\right)
    .
    \end{align*}
(By the inequality \eqref{eq:ieq:gam:1}, the function \(F_{r}(x)\) is well-defined on \(
        (0,\infty)\).)

Let \(j(x)\eqd \floor*{\lg\left( {x}/{\Gamma\left( \frac{r}{k} \right)} \right)+\alpha-\beta}\).
Then \(2^{j({x})+\beta-\alpha} \le {x}/{\Gamma\left( \frac{r}{k} \right)} <
    2^{j(x)+\beta-\alpha-1}\). In other words \( Q^{-1}\left(\frac{r}{k},2^{-i+\alpha-\beta}
    {x}/{\Gamma\left(\frac{r}{k}\right)}\right)=0 \) for \(i \le j(x)\). Thus
\begin{equation}\label{eq:F:x}
    \begin{aligned}
    F_{r}(x)
    &
    =
    \sum_{i \ge j(x)+1}
    2^{-i+\alpha-\beta}
    Q^{-1}\left(\frac{r}{k},2^{-i+\alpha-\beta}
        \frac{x}{\Gamma\left(\frac{r}{k}\right)}\right)
    \\
    &
    =
    \sum_{s \ge 1}
    2^{-s-j(x)+\alpha-\beta}
    Q^{-1}\left(\frac{r}{k},2^{-s-j(x)+\alpha-\beta}
        \frac{x}{\Gamma\left(\frac{r}{k}\right)}\right)
    \\
    &
    =
    \sum_{s \ge 1}
    x^{-1} \Gamma \left(\frac{r}{k}\right) 2^{\fracPart{\alpha -\beta +{\lg
    \left({x}/{\Gamma \left(\frac{r}{k}\right)}\right)}}-s} 
    Q^{-1}\left(\frac{r}{k},2^{\fracPart{\alpha -\beta +{\lg \left({x}/{\Gamma
    \left(\frac{r}{k}\right)}\right)}}-s}\right)
    \\
    &
    =
    \sum_{s \ge 1}
    x^{-1} \Gamma \left(\frac{r}{k}\right) 2^{\fracPart{\gamma+{\lg
    \left({x}/{\Gamma \left(\frac{r}{k}\right)}\right)}}-s} 
    Q^{-1}\left(\frac{r}{k},2^{\fracPart{\gamma+{\lg \left({x}/{\Gamma
    \left(\frac{r}{k}\right)}\right)}}-s}\right)
    ,
    \end{aligned}
\end{equation}
where the last step uses \eqref{eq:gamma}.  Note that \(F_{r}(x)\) is continuous and decreasing on \(
    (0, \infty)\), with \(F_{r}(x)\to 0\) as \(x \to \infty\).  By the derivative formula
\eqref{eq:Gam:D}, the derivative of \(F_{r}(x)\) is
\begin{equation}\label{eq:df}
    \frac{\mathrm d F_{r}(x)}{\mathrm d x}
    =
    -
    \sum_{s \ge 1}
    g_{r}(s,x)
    ,
\end{equation}
where
\begin{equation}\label{eq:df:1}
    \begin{aligned}
        g_{r}(s,x)
        \eqd
        &
        \frac{\Gamma \left(\frac{r}{k}\right)^2 }{x^2}
        4^{\fracPart{\gamma +{\lg \left({x}/{\Gamma \left(\frac{r}{k}\right)}\right)}}- s}
        \exp \left(Q^{-1}\left(\frac{r}{k},2^{\fracPart{\gamma +{\lg \left({x}/{\Gamma \left(\frac{r}{k}\right)}\right)}}-s}\right)\right)
        \\
        &
        Q^{-1}\left(\frac{r}{k},2^{\fracPart{\gamma +{\lg \left({x}/{\Gamma
                                \left(\frac{r}{k}\right)}\right)}}-s}\right)^{1-\frac{r}{k}} 
        .
    \end{aligned}
\end{equation}
Comparing with \eqref{eq:nurab}, we see that 
\begin{equation}\label{eq:d:nu}
    \frac{\dd \nurab}{\dd x}=\sum_{s \ge 1} g_{r}(s,x)
    \le
    \bigO{2^{\fracPart{\gamma+\lg(x /\Gamma\left( \frac{r}{k} \right)}}x^{-2} },
\end{equation}
and 
\(F_{r}(x)=\nu_{r,k,\gamma}(x,\infty)\), where the inequality follows from \eqref{eq:ieq:gam:1}.
Thus \autoref{lem:convergence} (\ref{lem:con:nu}) is proved.

\subsection{Proof of {\autoref{lem:convergence}} {(\ref{lem:con:mean})}}
\label{sec:con:mean}

Assume for now that \(h \eqd 2^{\beta -\alpha } \Gamma \left(\frac{r}{k}\right) < 1\).
Let \(j_{1}=\floor*{\alpha-\beta-\lg \Gamma \left(\frac{r}{k}\right)}\),
i.e., \(2^{j_{1}} h \le 1 < 2^{j_{1}+1} \).
By the upper bound of \(\dd \nurab/\dd x=\sum_{s\ge 1} g_{r}(s,x)\) in \eqref{eq:d:nu}, \(\int_{h}^{1}x \dd\nurab(x) <
    \infty\). Thus we are allowed to write this integral as
\begin{equation}\label{eq:split:nu:mean}
    \int_{h}^{1} x \dd \nu_{r,k,\gamma}\,(x)
    =
    \sum_{i=0}^{j_{1}}
    \int_{h 2^{i}}^{h 2^{i+1} \wedge 1} x 
    \,
    \dd \nu_{r,k,\gamma}(x)
    =
    \sum_{s \ge 1}
    \sum_{i=0}^{j_{1}}
    \int_{h 2^{i}}^{h 2^{i+1} \wedge 1} x 
    g_{r}(s,x)
    \,
    \dd x
    .
\end{equation}
For \(x \in (2^{i} h, 2^{i+1}h)\), by the definition of \(g_{r}(s,x)\) in \eqref{eq:df:1},
\begin{equation}\label{eq:f:g:s:x}
    \begin{aligned}
    &
    g_{r}(s,x)
    =
    \hat{g}_{r}(s,x,i)
    \eqd
    4^{ -i-s+\alpha -\beta} 
    \\
    &
    \qquad
    \exp\left\{Q^{-1}\left(\frac{r}{k},2^{-i-s+\alpha-\beta }
            \frac{x}{\Gamma \left(\frac{r}{k}\right)}\right)\right\}
    Q^{-1}\left(\frac{r}{k},2^{-i-s+\alpha-\beta }\frac{x}{\Gamma
            \left(\frac{r}{k}\right)}\right)^{1-\frac{r}{k}}
    .
    \end{aligned}
\end{equation}
Using the derivative formula \eqref{eq:Gam:D}, one can verify that
\begin{equation}\label{eq:g:diff}
    \frac{\partial}{\partial x}
    \hat{G}(s,x,i)
    =
    x
    \hat{g}_{r}(s,x,i)
    ,
\end{equation}
where
\begin{equation}\label{eq:g:hat}
    \begin{aligned}
    \hat{G}(s,x,i)
    \eqd
    &
    \Gamma \left(1+\frac{r}{k},Q^{-1}\left(\frac{r}{k},\frac{2^{-i-t+\alpha-\beta } x}{\Gamma
                \left(\frac{r}{k}\right)}\right)\right)
    \\
    &
    -x 2^{\alpha-\beta -i-t}
    Q^{-1}\left(\frac{r}{k},\frac{2^{-i-t+\alpha-\beta } x}{\Gamma \left(\frac{r}{k}\right)}\right)
    .
    \end{aligned}
\end{equation}
Therefore,
\begin{equation}\label{eq:g:int:i}
    \int_{h 2^{i}}^{h 2^{i+1}} x 
    g_{r}(s,x)
    \,
    \dd x
    =
    \int_{h 2^{i}}^{h 2^{i+1}} x 
    \hat{g}_{r}(s,x,i)
    \,
    \dd x
    =
    \hat{G}(s,h 2^{i+1},i)
    -
    \hat{G}(s,h 2^{i},i)
    .
\end{equation}
Summing \eqref{eq:g:int:i} over \(i\) and \(s\) as in \eqref{eq:split:nu:mean} and simplifying
through \cite[8.8.2]{DLMF}
\begin{equation}\label{eq:shift:gamma}
\Gamma\left(a+1,z\right)=a\Gamma\left(a,z\right)+z^{a}\ee^{-z},
\end{equation}
we have
\begin{equation}\label{eq:nu:mean:j1}
    \int_{h}^{1} x \dd \nu_{r,k,\gamma}\,(x)
    =
    \Gamma\left( 1+\frac{1}{k} \right) j_{1}
    + 
    \mu\left(\fracPart*{\gamma-\lg \Gamma \left(\frac{r}{k}\right)} \right)
    -
    \mu\left( 0 \right),
\end{equation}
where
\begin{equation}\label{eq:nu:def}
    \begin{aligned}
        \mu(x)
        \eqd
        &
        2^x \Gamma \left(1+\frac{r}{k}\right)
        +
        \sum_{s \ge 1} 
        \exp\left({-Q^{-1}\left(\frac{r}{k},2^{x-s}\right)}\right) Q^{-1}\left(\frac{r}{k},2^{x-s}\right)^{r/k}
        \\
        &
        -
        \sum_{s \ge 1} 
        2^{x-s} \Gamma \left(\frac{r}{k}\right) Q^{-1}\left(\frac{r}{k},2^{x-s}\right)
        .
    \end{aligned}
\end{equation}
By a similar argument, \eqref{eq:nu:mean:j1} also holds when \(h \ge 1\).
(When \(r=k\), \eqref{eq:nu:mean:j1} reduces to \(\lfloor \alpha -\beta \rfloor
    +2^{\fracPart{\alpha -\beta}}-1\), as in \cite{janson04}.)

We next compute \(\sum_{v \text{ good}} \E{\xi_{r,v} \mathbb 1 [\xi_{r,v}\le h]}\).
By definition, if \(v\) is good, then \(n_{v}=2^{m-t}-1\) with \(l/2 \le t \le L\).
Let \(u_{r,t}(x)\) be the probability density function of \(\xirv\).
Differentiating \eqref{eq:prob:k} shows that
uniformly for all \(t \le L\) and \(x \ge m^{-5}\),
\begin{equation}\label{eq:xi:den}
        u_{r,t}(x)
        =
        \left( 
            1+\bigO{\frac{\log m}{m}}^{\frac{1}{k}}
        \right)
        \hat{u}_{r,t}(x)
        ,
\end{equation}
where
\begin{equation}\label{eq:xi:den:1}
        \hat{u}_{r,t}(x)
        =
        \frac{n }
        {m^2 n_{v}}
        \exp\left( {Q^{-1}\left(\frac{r}{k},\frac{n x}{ \Gamma \left(\frac{r}{k}\right) m
                        n_{v}}\right)} \right) 
        Q^{-1}\left(\frac{r}{k},\frac{n x}{\Gamma \left(\frac{r}{k}\right) m
                n_{v}}\right)^{1-\frac{r}{k}}
        .
\end{equation}
Using again the derivative formula \eqref{eq:Gam:D}, one can verify that
\begin{equation}\label{eq:diff:u}
    \frac{\partial}{\partial x}
    \hat{U}_{r,t}(x)
    =
    x \hat{u}_{r,t}(x)
    ,
\end{equation}
where
\begin{equation}\label{eq:xi:int:0}
    \begin{aligned}
        &
        \hat{U}_{r,t}(x)
        \eqd
        \frac{n_{v} }{n}\Gamma \left(1+\frac{r}{k},Q^{-1}\left(\frac{r}{k},\frac{n x}{\Gamma
                    \left(\frac{r}{k}\right) m n_{v}}\right)\right)
        -
        \frac{x }{m}Q^{-1}\left(\frac{r}{k},\frac{n x}{\Gamma \left(\frac{r}{k}\right) m n_{v}}\right)
        .
    \end{aligned}
\end{equation}

Note also that \(\hat{u}_{r,t}(x)=0\) if 
\( \frac{n x}{ \Gamma \left({1}/{k}\right) m n_{v}} \ge 1\).
Thus 
\begin{equation}\label{eq:xi:int}
    \begin{aligned}
    &
    \E{\xirv \mathbb 1[\xirv \le h]}
    =
    \int_{m^{-5}}^{h} x u_{t}(x)
    \, \dd x
    +
    \E{\xirv \mathbb 1[\xirv \le m^{-5}]}
    \\
    &
    \qquad
    =
    \left( 
        1+\bigO{\frac{\log m}{m}}^{\frac{1}{k}}
    \right)
    \left(  
        \hat{U}_{r,t}\left( 
            \frac{\Gamma \left(\frac{r}{k}\right) m n_{v}}{n} \wedge h
        \right)
        -
        \hat{U}_{r,t}\left( 
            m^{-5}
        \right)
    \right)
    +
    \smallo{m^{-2}}
    \\
    &
    \qquad
    =
    \left( 
        1+\bigO{\frac{\log m}{m}}^{\frac{1}{k}}
    \right)
        \hat{U}_{r,t}\left( 
            \frac{\Gamma \left(\frac{r}{k}\right) m n_{v}}{n} \wedge h
        \right)
    +
    \smallo{m^{-2}}
    ,
    \end{aligned}
\end{equation}
where we use
\( \hat{U}_{r,t}(m^{-5}) = \smallo{m^{-2}},\)
which follows from the inequalities \eqref{eq:ieq:gam}, \eqref{eq:ieq:gam:1} and \eqref{eq:ieq:gam:2}.

If \(t \ge l+1\), then 
\begin{equation}\label{eq:high}
    \frac{n h}{\Gamma \left(\frac{r}{k}\right) m n_{v}}
    \ge
    \frac{2^{\beta-\alpha} \Gamma \left(\frac{r}{k}\right)}{\Gamma \left(\frac{r}{k}\right)}
    \frac{2^{m+\alpha+o(1)}}{2^{l+\beta+o(1)}2^{m-l-1}}
    =
    2^{1+o(1)} > 1
    ,
\end{equation}
for \(n\)
large. 
Thus \eqref{eq:xi:int} reduces to
\begin{equation}\label{eq:xi:int:1}
    \begin{aligned}
        \E{\xirv \mathbb 1[\xirv \le h]}
        &
        =
        \left( 
            1+\bigO{\frac{\log m}{m}}^{\frac{1}{k}}
        \right)
            \hat{U}_{r,t}\left( 
                \frac{\Gamma \left(\frac{r}{k}\right) m n_{v}}{n}
            \right)
        +
        \smallo{m^{-2}}
        \\
        &
        =
        (1+o(1))
        \frac{n_{v} }{n}\Gamma \left(1+\frac{r}{k}\right)
        +
        \smallo{m^{-2}}
        .
    \end{aligned}
\end{equation}
If \(t \le l\), then
\begin{equation}\label{eq:low}
    \frac{n h}{\Gamma \left(\frac{r}{k}\right) m n_{v}}
    \le
    \frac{2^{\beta-\alpha} \Gamma \left(\frac{r}{k}\right)}{\Gamma \left(\frac{r}{k}\right)}
    \frac{2^{m+\alpha+o(1)}}{2^{l+\beta+o(1)}2^{m-l}}
    =
    1^{1+o(1)}
    ,
\end{equation}
and \eqref{eq:xi:int} reduces to
\begin{equation}\label{eq:xi:int:2}
        \E{\xirv \mathbb 1[\xirv \le h]}
        =
        \left( 
            1+\bigO{\frac{\log m}{m}}^{\frac{1}{k}}
        \right)
            \hat{U}_{r,t}\left( 
                h
            \right)
        +
        \smallo{m^{-2}}
        .
\end{equation}
So we distinguish three cases, \(l/2 \le t \le l\), \(l < t < L\), and
\(t = L\), which we refer to as the \emph{low} part, the \emph{high} part, and the \emph{last} part.

The number of good nodes \(v\) with \(n_{v}=2^{m-t}-1\), is given by \eqref{eq:good:bad}. Thus
for the low part, i.e., when \(v\) is a good node with \(n_{v}=2^{m-t}\) and \(l/2 \le t \le
    l\),  
\begin{equation}\label{eq:mu:low}
    \begin{aligned}
        \mu_{1}
        &
        \eqd
        \sum_{v \text{ good and low}}
        \E{\xirv \mathbb 1[\xirv \le h]}
        \\
        &
        =
        \sum_{l/2 \le t \le l}
        \left( 2^{t+\alpha_{n}}+\bigO{1} \right)
        \left( 
        \left( 
            1+\bigO{\frac{\log m}{m}}^{\frac{1}{k}}
        \right)
            \hat{U}_{r,t}\left( 
                h
            \right)
        +
        \smallo{m^{-2}}
        \right)
        \\
        &
        =
            \underset{\text{s}=1}{\overset{\infty }{\sum }}\exp\left({-Q^{-1}\left(\frac{r}{k},2^{-\text{s}}\right)}  \right)
            Q^{-1}\left(\frac{r}{k},2^{-\text{s}}\right)^{\frac{r}{k}}
        \\
        &
        \qquad
            -\underset{\text{s}=1}{\overset{\infty }{\sum }}2^{-\text{s}} \Gamma \left(\frac{r}{k}\right) Q^{-1}\left(\frac{r}{k},2^{-\text{s}}\right)+2 \Gamma \left(1+\frac{r}{k}\right)
            + \smallo{1}
        ,
    \end{aligned}
\end{equation}
where the result has been simplified using \eqref{eq:shift:gamma}.
(The convergence of this sum follows from \eqref{eq:ieq:gam:1} and \eqref{eq:ieq:gam:2}.)
For the high part, i.e., when \(v\) is a good node with \(n_{v}=2^{m-t}\) and \(l < t < L\),
\begin{equation}\label{eq:mu:high}
    \begin{aligned}
        \mu_{2}
        &
        \eqd
        \sum_{v \text{ good and high}}
        \E{\xirv \mathbb 1[\xirv \le h]}
        \\
        &
        =
        \sum_{l < t < L}
        \left( 2^{t+\alpha_{n}}+\bigO{1} \right)
        \left( 
        \left(1+o(1)\right) 
        \frac{n_{v} }{n}\Gamma \left(1+\frac{r}{k}\right)
        +
        \smallo{m^{-2}}
        \right)
        \\
        &
        =
        \Gamma \left(1+\frac{r}{k}\right) (L-l-1)
        +
        \smallo{1}
        .
    \end{aligned}
\end{equation}
And for the last part, i.e., when \(v\) is good node with \(n_{v}=2^{m-L}\),
\begin{equation}\label{eq:mu:last}
    \begin{aligned}
        \mu_{3}
        &
        \eqd
        \sum_{v \text{ good and last}}
        \E{\xirv \mathbb 1[\xirv \le h]}
        \\
        &
        =
        \left( 
        (2-2^{\alpha_{n}}) 2^{L} + \bigO{1}
        \right)
        \left( 
        \left(1+o(1)\right) 
        \frac{n_{v} }{n}\Gamma \left(1+\frac{r}{k}\right)
        +
        \smallo{m^{-2}}
        \right)
        \\
        &
        =
        2^{-\alpha } \left(2-2^{\alpha }\right) \Gamma \left(1+\frac{r}{k}\right)
        +
        \smallo{1}
        .
    \end{aligned}
\end{equation}
Together with \eqref{eq:nu:mean:j1},
\begin{equation}\label{eq:mu:all}
    \begin{aligned}
    &
    \sum_{v \text{ good}} \E{\xi_{r,v} \mathbb 1 [\xi_{r,v}\le h]} 
    \\
    &
    =
    \mu_{1}+\mu_{2}+\mu_{3}+\smallo{1}
    \\
    &
    \to
    f_{r,k,\gamma}
    +
    \Gamma\left( 1+\frac{r}{k} \right)
    \left( 
        2^{1-\alpha }+\alpha -\beta -\ell + L
    \right)
    -
    \int_{h}^{1} x \,\dd \nu_{r,k,\gamma}(x)
    ,
    \end{aligned}
\end{equation}
where
\begin{equation}\label{eq:f:gamma}
    \begin{aligned}
        f_{r,k,\gamma}
        \eqd
        &
        \underset{t=1}{\overset{\infty }{\sum }}\exp \left(-Q^{-1}\left(\frac{r}{k},2^{\fracPart{\gamma-\lg\Gamma(\frac{r}{k})}-t}\right)\right) Q^{-1}\left(\frac{r}{k},2^{\fracPart{\gamma-\lg\Gamma(\frac{r}{k})}-t}\right)^{\frac{r}{k}}
        \\
        &
        +
        \underset{t=1}{\overset{\infty }{\sum }}-2^{\fracPart{\gamma-\lg\Gamma(\frac{r}{k})}-t} \Gamma \left(\frac{r}{k}\right) Q^{-1}\left(\frac{r}{k},2^{\fracPart{\gamma-\lg\Gamma(\frac{r}{k})}-t}\right)
        \\
        &
        +\Gamma \left(1+\frac{r}{k}\right) 
        \left( 
            2^{\fracPart{\gamma-\lg\Gamma(\frac{r}{k})}}
            -\fracPart{\gamma-\lg\Gamma(\frac{r}{k})}
            -\lg \Gamma \left(\frac{r}{k}\right)-1
        \right)
        .
    \end{aligned}
\end{equation}
(The fact that \(f_{r,k,\gamma}<\infty\) follows from the inequalities \eqref{eq:ieq:gam},
\eqref{eq:ieq:gam:1}, \eqref{eq:ieq:gam:2}.)
When \(k=r\), the above is simply \(2^{\gamma }-\gamma -1\), as in Theorem
1.1 of \cite{janson04}.

\subsection{Proof of {\autoref{lem:convergence}} {(\ref{lem:con:var})}}
\label{sec:con:var}

By the upper bound of \(\dd \nurab/\dd x=\sum_{s\ge 1} g_{r}(s,x)\) in \eqref{eq:d:nu},
\(\int_{0}^{h}x^2 \dd\nurab(x) < \infty\). Thus we are allowed to write this integral as
\begin{equation}\label{eq:split:nu:var}
    \int_{0}^{h} x^{2} \dd \nu_{r,k,\gamma}\,(x)
    =
    \sum_{i=-1}^{-\infty}
    \int_{h 2^{i}}^{h 2^{i+1}} x^{2}
    \,
    \dd \nu_{r,k,\gamma}(x)
    =
    \sum_{s \ge 1}
    \sum_{i=-1}^{-\infty}
    \int_{h 2^{i}}^{h 2^{i+1}} x^{2}
    g_{r}(s,x)
    \,
    \dd x
    .
\end{equation}
Recall that for \(x \in (2^{i}h, 2^{i+1}h)\), \(g_{r}(s,x)=\hat{g}_{r}(s,x,i)\) (see
\eqref{eq:f:g:s:x}).
Using the derivative formula \eqref{eq:Gam:D}, one can verify that
\begin{equation}\label{eq:g:diff:2}
    \frac{\partial}{\partial x}
    \widetilde{G}_{r}(s,x,i) 
    =
    x^{2} \hat{g}_{r}(s,x,i)
    ,
\end{equation}
where
\begin{equation}\label{eq:G:2}
\begin{aligned}
    \widetilde{G}_{r}(s,x,i) 
    \eqd
    &
    2 x \exp \left( {-Q^{-1}\left(\frac{r}{k},\frac{2^{-i-t+\alpha -\beta } x}{\Gamma \left(\frac{r}{k}\right)}\right)} \right) Q^{-1}\left(\frac{r}{k},\frac{2^{-i-t+\alpha -\beta } x}{\Gamma \left(\frac{r}{k}\right)}\right)^{r/k}
    \\
    &
    -x^2 2^{\alpha -\beta -i-t} Q^{-1}\left(\frac{r}{k},\frac{2^{-i-t+\alpha -\beta } x}{\Gamma \left(\frac{r}{k}\right)}\right)
    \\
    &
    +\frac{2 r x }{k}\Gamma \left(\frac{r}{k},Q^{-1}\left(\frac{r}{k},\frac{2^{-i-t+\alpha -\beta } x}{\Gamma \left(\frac{r}{k}\right)}\right)\right)
    \\
    &
    -2^{-\alpha +\beta +i+\frac{k-2 r}{k}+t} \Gamma \left(\frac{2 r}{k},2 Q^{-1}\left(\frac{r}{k},\frac{2^{-i-t+\alpha -\beta } x}{\Gamma \left(\frac{r}{k}\right)}\right)\right)
    \\
    &
    -\frac{r 2^{-\alpha +\beta +i+t} }{k}\Gamma \left(\frac{r}{k},Q^{-1}\left(\frac{r}{k},\frac{2^{-i-t+\alpha -\beta } x}{\Gamma \left(\frac{r}{k}\right)}\right)\right)^2
    .
\end{aligned}
\end{equation}
Thus
\begin{equation}\label{eq:G:2:i}
    \int_{h 2^{i}}^{h 2^{i+1}} x^{2}
    g_{r}(s,x)
    \,
    \dd x
    =
    \int_{h 2^{i}}^{h 2^{i+1}} x^{2}
    \hat{g}_{r}(s,x,i)
    \,
    \dd x
    =
    \widetilde{G}_{r}(s, h 2^{i+1}, i)
    -
    \widetilde{G}_{r}(s, h 2^{i}, i)
    .
\end{equation}
Summing \eqref{eq:G:2:i} over \(i\) and \(s\) as in \eqref{eq:split:nu:var}
\begin{equation}\label{eq:split:nu:var:1}
    \begin{aligned}
        \int_{0}^{h} x^{2} \dd \nu_{r,k,\gamma}\,(x)
        =
        2^{\beta -\alpha }
        &
        \left(
            \frac{3 r \Gamma \left(\frac{r}{k}\right)^2}{k}-4^{1-\frac{r}{k}} \Gamma \left(\frac{2 r}{k}\right)
            -\underset{t=1}{\overset{\infty }{\sum }}2^{-t} \Gamma \left(\frac{r}{k}\right)^2 Q^{-1}\left(\frac{r}{k},2^{-t}\right)
        \right.
        \\
        &
            +2 \underset{t=1}{\overset{\infty }{\sum }} \exp\left( {-Q^{-1}\left(\frac{r}{k},2^{-t}\right)} \right) \Gamma \left(\frac{r}{k}\right) Q^{-1}\left(\frac{r}{k},2^{-t}\right)^{r/k}
        \\
        &
        \left.
            -2^{1-\frac{2 r}{k}} \underset{t=1}{\overset{\infty }{\sum }}2^t \Gamma \left(\frac{2 r}{k},2 Q^{-1}\left(\frac{r}{k},2^{-t}\right)\right)
        \right)
        .
    \end{aligned}
\end{equation}
(The convergence of this sum follows from \eqref{eq:ieq:gam:1} and \eqref{eq:ieq:gam:2}.)
Note that this is simply \(3 \cdot 2^{\beta-\alpha-1}\) for \(k=r\), as in Lemma 2.5 of \cite{janson04}.

We next compute \(\sum_{v \text{ good}} \V{\xi_{r,v} \mathbb 1 [\xi_{r,v}\le h]}\).
Using the estimation \eqref{eq:xi:int:1} and \eqref{eq:xi:int:2}, we see that
\begin{equation}\label{eq:sum:xi:2}
    \sum_{v \text{ good}}
    \E{\xirv \mathbb 1[\xirv \le h]}^{2}
    =
    o(1)
    .
\end{equation}
Thus it suffices to 
compute \(\sum_{v \text{ good}} \E{\xi_{r,v}^{2} \mathbb 1 [\xi_{r,v}\le h]}\).

Let \(v\) be a good node with \(n_{v}=2^{m-t}\) and \(l/2 \le t \le l\). Then using
\eqref{eq:xi:den}, 
\begin{equation}\label{eq:xi:var}
    \begin{aligned}
        \E{\xirv^{2} \mathbb 1[\xirv \le h]}
        &
        =
        \int_{m^{-5}}^{h} x^{2} u_{t}(x)
        \, \dd x
        +
        \smallo{m^{-2}}
        \\
        &
        =
        \left(1+\bigO{\frac{\log m}{m}}^{\frac{1}{k}}\right) 
        \int_{m^{-5}}^{h} x^{2} \hat{u}_{r,t}(x)
        \, \dd x
        +
        \smallo{m^{-2}}
        .
    \end{aligned}
\end{equation}
Using again the derivative formula \eqref{eq:Gam:D}, one can verify that
\begin{equation}\label{eq:diff:u:2}
    \frac{\partial}{\partial x} \widetilde{U}_{r,t}(x)
    =
    x^{2}
    \hat{u}_{r,t}(x)
    ,
\end{equation}
where
\begin{equation}\label{eq:int:U:2}
    \begin{aligned}
        \widetilde{U}_{r,t}(x)
        \eqd
        &
        \frac{2 x n_{v} }{n}\exp\left({-Q^{-1}\left(\frac{r}{k},\frac{n x}{m \Gamma \left(\frac{r}{k}\right) n_{v}}\right)} \right) Q^{-1}\left(\frac{r}{k},\frac{n x}{m \Gamma \left(\frac{r}{k}\right) n_{v}}\right)^{r/k}
        \\
        &
        -\frac{m r n_{v}^2}{k n^2} \Gamma \left(\frac{r}{k},Q^{-1}\left(\frac{r}{k},\frac{n x}{m \Gamma \left(\frac{r}{k}\right) n_{v}}\right)\right)^2
        -\frac{x^2}{m} Q^{-1}\left(\frac{r}{k},\frac{n x}{m \Gamma \left(\frac{r}{k}\right) n_{v}}\right)
        \\
        &
        -\frac{m 2^{\frac{k-2 r}{k}} n_{v}^2}{n^2} \Gamma \left(\frac{2 r}{k},2 Q^{-1}\left(\frac{r}{k},\frac{n x}{m \Gamma \left(\frac{r}{k}\right) n_{v}}\right)\right)
        \\
        &
        +\frac{2 r x n_{v}}{k n} \Gamma \left(\frac{r}{k},Q^{-1}\left(\frac{r}{k},\frac{n x}{m \Gamma \left(\frac{r}{k}\right) n_{v}}\right)\right)
        .
    \end{aligned}
\end{equation}
Recall that \(\hat{u}_{r,t}(x)=0\) if 
\( \frac{n x}{ \Gamma \left({1}/{k}\right) m n_{v}} \ge 1\).
Thus 
\begin{equation}\label{eq:u:int:var}
    \begin{aligned}
        &
        \E{\xirv^{2} \mathbb 1[\xirv \le h]}
        =
        \int_{m^{-5}}^{h} x^{2} u_{t}(x)
        \, \dd x
        +
        \E{\xirv \mathbb 1[\xirv \le m^{-5}]}
        \\
        &
        \qquad
        =
        \left( 
            1+\bigO{\frac{\log
                    m}{m}}^{\frac{1}{k}}
        \right)
        \left(  
            \widetilde{U}_{r,t}\left( 
                \frac{\Gamma
                    \left(\frac{r}{k}\right)
                    m
                    n_{v}}{n}
                \wedge
                h
            \right)
            -
            \widetilde{U}_{r,t}\left( 
                m^{-5}
            \right)
        \right)
        +
        \smallo{m^{-2}}
        \\
        &
        \qquad
        =
        \left( 
            1+\bigO{\frac{\log
                    m}{m}}^{\frac{1}{k}}
        \right)
        \widetilde{U}_{r,t}\left( 
            \frac{\Gamma
                \left(\frac{r}{k}\right)
                m
                n_{v}}{n}
            \wedge
            h
        \right)
        +
        \smallo{m^{-2}}
        ,
    \end{aligned}
\end{equation}
where we use
\( \widetilde{U}_{r,t}(m^{-5}) = \smallo{m^{-2}},\)
which follows from the inequalities \eqref{eq:ieq:gam}, \eqref{eq:ieq:gam:1} and \eqref{eq:ieq:gam:2}.

The number of good nodes \(v\) with \(n_{v}=2^{m-t}-1\), is given by \eqref{eq:good:bad}.  We again
separate good nodes into the low part (\(l/2 \le t \le l\)), the high part (\(l < t < L\)) and the
last part (\(t = L\)) as in \autoref{sec:con:mean}.  
For the low part, i.e., when \(v\) is a good node with \(n_{v}=2^{m-t}\) and \(l/2 \le t \le
    l\),  
\begin{equation}\label{eq:sigma:low}
    \begin{aligned}
        \sigma_{1}
        &
        \eqd
        \sum_{v \text{ good and low}}
        \E{\xirv^{2} \mathbb 1[\xirv \le h]}
        \\
        &
        =
        \sum_{l/2 \le t \le l}
        \left( 2^{t+\alpha_{n}}+\bigO{1} \right)
        \left(1+\bigO{\frac{\log m}{m}}^{\frac{1}{k}}\right) 
        \left( 
            \widetilde{U}_{r,t}\left( 
                h
            \right)
            +
            o\left( m^{-2} \right)
        \right)
        \\
        &
        =
        2^{\beta -\alpha } \left(
        \underset{\text{s}=1}{\overset{\infty }{\sum }}2 \exp\left({-Q^{-1}\left(\frac{r}{k},2^{-\text{s}}\right)} \right) \Gamma \left(\frac{r}{k}\right) Q^{-1}\left(\frac{r}{k},2^{-\text{s}}\right)^{r/k}
        \right.
        \\
        &
        \qquad
        \qquad
            \frac{2 r \Gamma \left(\frac{r}{k}\right)^2}{k}-2^{1-\frac{2 r}{k}} \Gamma \left(\frac{2 r}{k}\right)
            -\underset{\text{s}=1}{\overset{\infty }{\sum }}2^{-\text{s}} \Gamma \left(\frac{r}{k}\right)^2 Q^{-1}\left(\frac{r}{k},2^{-\text{s}}\right)
        \\
        &
        \qquad
        \qquad
        \left.
            -\underset{\text{s}=1}{\overset{\infty }{\sum }}2^{-\frac{2 r}{k}+\text{s}+1} \Gamma \left(\frac{2 r}{k},2 Q^{-1}\left(\frac{r}{k},2^{-\text{s}}\right)\right)
        \right)
        .
    \end{aligned}
\end{equation}
For the high part, i.e., when \(v\) is a good node with \(n_{v}=2^{m-t}\) and \(l < t < L\),
\begin{equation}\label{eq:sigma:high}
    \begin{aligned}
        \sigma_{2}
        &
        \eqd
        \sum_{v \text{ good and high}}
        \E{\xirv^{2} \mathbb 1[\xirv \le h]}
        \\
        &
        =
        \sum_{l < t < L}
        \left( 2^{t+\alpha_{n}}+\bigO{1} \right)
        \left(1+\bigO{\frac{\log m}{m}}^{\frac{1}{k}}\right) 
        \left( 
            \widetilde{U}_{r,t}\left( 
                \frac{\Gamma \left(\frac{r}{k}\right) m n_{v}}{n}
            \right)
            +
            o\left( m^{-2} \right)
        \right)
        \\
        &
        =
        2^{\beta -\alpha } \left(\frac{r \Gamma \left(\frac{r}{k}\right)^2}{k}-2^{1-\frac{2 r}{k}} \Gamma \left(\frac{2 r}{k}\right)\right)
        +o(1)
        .
    \end{aligned}
\end{equation}
And for the last part, i.e., when \(v\) is good node with \(n_{v}=2^{m-L}\),
\begin{equation}\label{eq:sigma:last}
    \begin{aligned}
        \sigma_{3}
        &
        \eqd
        \sum_{v \text{ good and last}}
        \E{\xirv^{2} \mathbb 1[\xirv \le h]}
        \\
        &
        =
        \left( 
        (2-2^{\alpha_{n}}) 2^{L} + \bigO{1}
        \right)
        \left(1+\bigO{\frac{\log m}{m}}^{\frac{1}{k}}\right) 
        \left( 
            \widetilde{U}_{r,t}\left( 
                \frac{\Gamma \left(\frac{r}{k}\right) m n_{v}}{n}
            \right)
            +
            o\left( m^{-2} \right)
        \right)
        \\
        &
        =o(1)
        .
    \end{aligned}
\end{equation}
Therefore,
\begin{equation}\label{eq:sigma:final}
    \sum_{v:h(v)\le L} 
    \V{\xi_{r,v} \mathbb 1 [\xi_{r,v}\le h]} 
    =
    \sigma_{1}+\sigma_{2}+\sigma_{3} + o(1)
    \to \int_{0}^{h} x^2 \, \dd \nu_{r,k,\gamma}(x),
\end{equation}
where the limit is given by \eqref{eq:split:nu:var:1}.
Thus we have completed the whole proof of \autoref{lem:convergence}.

\printbibliography{}

\end{document}